\documentclass[11pt]{amsart}
\usepackage{amsmath,amsthm,amsfonts,amssymb,,mathabx}
\usepackage{mathrsfs}
\usepackage{enumitem}
\usepackage[utf8]{inputenc}
\usepackage[OT1]{fontenc}
\usepackage{appendix}
\usepackage{comment}
\usepackage{nomencl}
 \usepackage{hyperref}
\hypersetup{  
  bookmarks=true,
  backref=true,
  pagebackref=false,
  colorlinks=true,
  linkcolor=blue,
  citecolor=red,
  urlcolor=blue
}

\usepackage[english]{babel}
\usepackage{amsbsy,amscd,fancyhdr,graphicx,psfrag,fancybox,indentfirst,color}
\usepackage{graphics,epsfig}
\usepackage{fullpage}
\usepackage{url}
\usepackage[numbers]{natbib}

\newtheorem{thm}{Theorem}[section]

\newtheorem{lem}[thm]{Lemma}
\newtheorem{prop}[thm]{Proposition}
\newtheorem{ex}[thm]{Example}%\theoremstyle{definition}
\newtheorem{defn}[thm]{Definition}
\newtheorem{rem}[thm]{Remark}%\numberwithin{equation}{section}

\newcommand{\norm}[1]{\|#1\|}
\newcommand{\normL}[1]{\big\|#1\big\|}
\newcommand{\norma}[1]{\left\|#1\right\|}
\newcommand{\abs}[1]{\left\vert#1\right\vert}
\newcommand{\babs}[1]{\big\vert#1\big\vert}
\newcommand{\pa}[1]{\left(#1\right)}
\newcommand{\bpa}[1]{\big(#1\big)}

\newcommand{\dotp}[2]{\langle #1,\,#2 \rangle}
\newcommand{\set}[1]{\left\{#1\right\}}
\newcommand{\grad}{{\mathrm{grad}}}
\newcommand{\To}{\longrightarrow}

\newcommand{\M}{\mathcal{M}}
\newcommand{\R}{\mathbb R}

\newcommand{\N}{\mathbb N}

\newcommand{\rmD}{{\rm D}}
\newcommand{\supp}{\mathrm{supp}}

\newcommand{\Sg}{\supp(\eta)}

\newcommand{\En}{\mathcal{E}_n}
\newcommand{\LIP}[1]{\mathrm{Lip}(#1)}
\newcommand{\Lip}[1]{L_{#1}}
\newcommand{\diam}{\mathrm{diam}}
\newcommand{\cut}{\mathrm{Cut}}
\newcommand{\inj}{\mathrm{inj}}
\newcommand{\Exp}{\mathrm{Exp}}
\newcommand{\ddt}{\displaystyle\frac{\partial}{\partial t}}
\newcommand{\NeigGama}{\mathcal{N}_{\Gamma}^{a_0}}

\newcommand{\proj}{\mathrm{Proj}}
\newcommand{\ball}{B}
\newcommand{\distH}{d_{\mathrm{H}}^\M}

\newcommand{\vol}{\mathrm{vol}}
\newcommand{\qandq}{\quad \text{and} \quad}
\newcommand{\qwithq}{\quad \text{with} \quad}

\renewcommand{\Tilde}{\widetilde}

\def\1{\mathbb I}
\def\a{\alpha}
\def\b{\beta}
\def\d{\delta}
\def\e{\varepsilon}
\def\g{\gamma}

\def\G{\Gamma}

\def\Mt{(\mathcal{M} \setminus \G) \times ]0,T[}

\def\tMt{(\widetilde{\mathcal{M}} \setminus \widetilde{\G}) \times ]0,T[}
\def\tGt{(\widetilde{\G} \times ]0,T[) \cup \widetilde{\mathcal{M}} \times \{0\}}

\newcommand{\eqdef}{\stackrel{\mathrm{def}}{=}}

\makenomenclature 
\date{}
\begin{document}
\title{Limits of non-local approximations to the Eikonal equation on Manifolds}
% \printnomenclature 
\address{{\sc Rita Zantout}: INSA Rouen Normandie, Normandie Univ, LMI UR 3226, F-76000 Rouen, France.}
\email{rita.zantout@insa-rouen.fr}

\address{{\sc Nicolas Forcadel}: INSA Rouen Normandie, Normandie Univ, LMI UR 3226, F-76000 Rouen, France.}
\email{nicolas.forcadel@insa-rouen.fr}

\address{{\sc Jalal Fadili}: Normandie Univ, ENSICAEN, CNRS, GREYC, France.}
\email{Jalal.Fadili@greyc.ensicaen.fr}

\author{Jalal Fadili \and Nicolas Forcadel \and Rita Zantout}

\maketitle

\begin{abstract}
In   this paper, we consider a non-local approximation of the time-dependent Eikonal equation defined on a Riemannian manifold. We show that the local and the non-local problems are well-posed in the sense of viscosity solutions and we prove regularity properties of these solutions in time and space. If the kernel is properly scaled, we then derive error bounds between the solution to the non-local problem and the one to the local problem, both in continuous-time and Forward Euler discretization. Finally, we apply these results to a sequence of random weighted graphs with $n$ vertices. In particular, we establish that the solution to the problem on graphs converges almost surely uniformly to the viscosity solution of the local problem as the kernel scale parameter decreases at an appropriate rate when the number of vertices grows and the time step vanishes.
\end{abstract}

\tableofcontents

\section{Introduction}
Nonlinear partial differential equations (PDEs) on graphs have found applications in a variety of areas such as, e.g.,   analysis, physics, economy, probability theory, biology and data science. In particular,  a family of Hamilton-Jacobi equations on graphs called the Eikonal equation has been considered on weighted graphs for data processing in \cite{desquesnes2013eikonal,ta2009adaptation}, for semi-supervised learning on graphs \cite{calder2022hamilton,roith2022continuum}, and for data depth \cite{calder2022hamilton,molina2022tukey}. They have been also used on topological networks in \cite{camilli2013approximation,imbert2013hamilton}. In \cite{shu2018hamilton}, Hamilton-Jacobi equations on graphs were also studied to derive discrete versions of some functional inequalities (log-Sobolev inequality and Talagrand's transport inequality).  

In this paper, we are particularly interested in the case of geometric graphs whose vertices are points living on a compact Riemannian submanifold, which is relevant for many applications such as data depth \cite{molina2022eikonal}, semi-supervised learning \cite{belkin2004semi} and image and mesh processing \cite{memoli2005distance,osher2017low}.

% There has been recently a wave of interest in adapting and solving PDEs on data represented by arbitrary graphs and networks. Using this framework, problems are directly expressed in a discrete setting where an appropriate discrete differential calculus can be proposed.
In order to transpose PDEs on graph, discrete calculus have been used in recent years using partial difference equations (PdEs) on graphs. PdEs are methods used to reformulate continuous problems by replacing differential operators by difference operators on graphs~\cite{elmoataz2008nonlocal,elmoataz2008unifying,grady2010discrete}.
 
\medskip

The main goal of this paper is to rigorously study continuum limits, i.e. as the number of data points tends to infinity, of the Eikonal equation defined on a weighted geometric graph embedded in a compact Riemannian submanifold. The motivation behind considering Eikonal equations in such a context is the ability to extend it to any discrete data that can be represented by weighted geometric graphs. In fact, many applications in numerical data analysis and processing or machine learning include data that can be defined on manifolds, or irregularly shaped domains, or network-like structures, or defined as high dimensional point clouds such as collections of features vectors. In a discrete setting, these data can be represented as weighted geometric graphs, where the vertices are drawn from the underlying domain (a manifold) and are connected by edges if sufficiently close in a certain ground metric. The edges are given weights (e.g., based on the distance between data points).

\subsection{Problem statement}

\nomenclature{$\M$}{differentiable manifold}
\nomenclature{$g$}{Riemannian metric}
\nomenclature{$\norm{.}_x=\sqrt{g_x(.,.)}$}{Riemannian inner product norm }
\nomenclature{$C^1$}{continuously differentiable}
\nomenclature{$C^k$}{k times continuously differentiable}
\nomenclature{$T_x\M$}{tangent space at $x \in M$}
\nomenclature{$\R$}{real numbers}
\nomenclature{$\nabla_v$}{directional derivative }
\nomenclature{$T^*\M$}{cotangent bundle}
\nomenclature{$T^*_x\M$}{cotangent space}
\nomenclature{$\ddt$}{partial derivative with respect to time}
\nomenclature{$H$}{Hamiltonian}
\nomenclature{$P$}{Potential function }
\nomenclature{$L(\gamma)$}{length of the curve $\gamma$}
\nomenclature{$L_f$}{lipschitz constant to the function $f$}
\nomenclature{$d_{\M}$}{Riemannian distance}
\nomenclature{$\mathcal{L}_{x y}$}{parallel transport from $T_x\M$ to $T_y\M$}
\nomenclature{$\LIP$}{lipschtiz function}
\nomenclature{$\nabla ^{-}_{\eta_\e}$}{weighted directional internal gradient operator}
\nomenclature{$\eta$}{kernel function}
\nomenclature{$\eta_\e$}{$\e$-scaled kernel function}
\nomenclature{$\R^+$}{positive real numbers}
\nomenclature{$\supp$}{support of a function}
\nomenclature{$Proj$}{projection}
\nomenclature{$|.|$}{the absolute-value norm}
\nomenclature{$G_n$}{weighted graph}
\nomenclature{$V_n$}{the set of $n$ vertices}
\nomenclature{$E_n$}{the set of $n$ edges}
\nomenclature{$w_n$}{the weights}
\nomenclature{$\mu$}{probability measure}
\nomenclature{$\rho$}{density of $\mu$}
\nomenclature{$\diam$}{diameter}
\nomenclature{$B_\M(x,r)$}{geodesic ball of center $x$ and radius $r>0$}
\nomenclature{$H_d$}{discrete Hamiltonian}

In this paper, we will work with a Riemannian manifold $(\M,g)$ of dimension $m^*$, where $\M$ is a compact manifold and $g$ is a Riemannian metric (see Section~\ref{chap4_sec:def-manifold} for precise definitions and properties). Let $G = (V,w)$ be a finite weighted (geometric) graph on $\M$, where $V \subset \M$ is the set of vertices, and $w: V\times V \to \R^+$ is the weight function. A natural Eikonal-type equation on graphs takes the form
\begin{equation}\label{chap4_eq:eikgraph}
\begin{cases}
    \max_{v \in V} \sqrt{w(u,v)}(f(v)-f(u))_{-}=\widetilde{P}(u), & u\in V\setminus V_0,\\
    f(u)=0, & u \in V_0,
    \end{cases}
\end{equation}
where $(\cdot)_- \eqdef -\min(\cdot,0)$, $V_0 \subset V$ and $\widetilde{P}$ is a given potential. This equation is an adaptation on weighted graphs of the Eikonal equation using the framework of PdEs and provides a tool for multiple front propagation problems on weighted graphs. This discrete form allows to handle any data that can be represented on graphs, e.g., in machine learning, data analysis and processing on graphs, on unstructured meshes or point clouds  \cite{calder2022hamilton,camilli2013approximation,desquesnes2013eikonal,memoli2005distance,roith2022continuum,ta2009adaptation}.

% We should note that when $M$ is compact then $M$ is finite-dimensional, complete and uniformly locally convex. 
Our goal in the paper is to study the behavior of the solution to problem \eqref{chap4_eq:eikgraph} as the number of vertices goes to infinity. In fact, we will consider an even more general class of equations. More precisely, inspired by \eqref{chap4_eq:eikgraph}, we consider the non-local Eikonal equation in a time-dependent form
\begin{equation} \tag{\textrm{$\mathcal{P_{\e}}$}} \label{chap4_eikonal-eq-discrete}
\begin{cases}
\ddt f^\e(x,t) +  \abs{\nabla ^{-}_{\eta_\e}f^\e(x,t)}_\infty = \widetilde{P}(x), & (x,t) \in \tMt, \\
f^\e(x,t) = f^\e_0(x), & (x,t) \in \tGt,
\end{cases}
\end{equation}
where $\widetilde\M$ is a subset of points of $\M$, $\widetilde{\Gamma} \subset \widetilde\M$ is the set of boundary points, $\widetilde{P}$ is the potential function, and $f_0^\e$ is the boundary function. $\nabla ^{-}_{\eta_\e}$ is a non-local operator coined the weighted directional internal gradient operator, introduced in \cite{desquesnes2013eikonal} and studied theoretically in \cite{fadili2023limits} in the Euclidean case. This operator is defined through 
\begin{equation*}
   \abs{ \nabla ^{-}_{\eta_\e} f^\e(x,t) }_{\infty}= \max_{y \in \widetilde\M} J_\e(x,y)(f^\e(y,t)-f^\e(x,t)),
\end{equation*}
where, given a length scale $\e>0$, the $\e$-scaled kernel function $J_{\e}:\M \times \M \To \R_{+}$ is defined by
\begin{equation*}
  J_\e(x,y)= \frac{1}{C_\eta}  \eta_\e(\widetilde{d}(x,y)) \qwithq \eta_\e(t)=\frac{1}{\e}\eta \pa{\frac{t}{\e}},
\end{equation*}
and $\eta:[0, +\infty) \to [0,+\infty)$ is a radial kernel function, and 
\begin{equation} \label{chap4_def:C_eta}
    C_\eta= \sup_{t\in \R_{+}} t\eta(t) > 0.
\end{equation}
Since, from a practical point of view, computing intrinsic Riemannian distance is quite impossible in many cases due to the unknown geometry and curvature, we will work with $\widetilde{d}$ which is an approximation of the intrinsic distance function (see Assumption \ref{chap4_assum:tilde-d} below). Observe that we can also write
\begin{equation} \label{chap4_gradientdiscret}
  \abs{  \nabla ^{-}_{\eta_\e} f^\e(x,t)}_\infty= \max_{y \in \widetilde{M}} C_{\eta}^{-1} \eta_\e (\widetilde{d}(x,y)) (f^\e(x,t)-f^\e(y,t)).
\end{equation}

The problem \eqref{chap4_eikonal-eq-discrete} represents an Eikonal equation on weighted graphs with $n$ vertices when we properly instantiate the sets $\widetilde{\M}$ and $\widetilde{\Gamma}$, see Section~\ref{chap4_sec:eikconvgraphs}.
% Since, from a practical point of view, computing intrinsic Riemannian distance is quite impossible in many cases due to the unknown geometry and curvature. We were inspired by the work of Mémoli and Sapiro where they provide approximations of the intrinsic distance function to a given submanifold of $\R^m$ by the Euclidean distance function computed in a thin offset band that surrounds this manifold. More details can be found in \cite{memoli2005distance} and in Section \eqref{chap4_sec:def-manifold}.
% In particular, $\widetilde{d}(x,y)$ represents the
% Euclidean distance in the $\Bar{\xi}$-offset $\Omega_\M^{\Bar{\xi}}$ which denotes the set $\{x\in \R^m:d(\M,x) \leq \Bar{\xi}\}$. For the completeness of this paper, we present the result, provided by Mémoli and Sapiro, by the following theorem 
% \begin{thm}[Approximation of the intrinsic distance by the Euclidean one in an offset band of $\M$] \cite[Theorem 5]{memoli2005distance}\label{chap4_th:dis-appro}
% Let $\M$ be a $C^3$ Riemannian manifold with $C^1$ boundary $\partial\M$, and assuming that $\M$ is geodesically strongly convex, we have that for small enough $\e>0$ and assuming $\bar{\xi}=\e^{2+2\xi}$,
% \begin{equation}\label{chap4_eq:dis-appro}
%     \max_{(x,y) \in \M\times\M}  \left|  \widetilde{d} (x,y) - d_\M(x,y)\right| \leq  C_\M  \sqrt{\bar{\xi}}=C_\M \e^{1+\xi}, 
%  \end{equation}
% where the constant $C_\M$ does not depend on $\bar{\xi}$. This result remains valid if $\M$ is a compact $C^2$ submanifold of $\R^m$ with $\partial\M = \emptyset $.
% \end{thm}
Therefore, several natural questions arise from a numerical analysis perspective: (i) given $n$-dependent scaling $\e_n$, is there a continuum limit (and in which sense) of the solution to \eqref{chap4_eikonal-eq-discrete} on graphs (and its time-discretized version as well) as $n \to +\infty$ (and time step goes $0$)? (ii) at which rate this convergence happens? (iii) what are the main quantities that come into play in the error bounds? Our main contributions of this work is to settle these questions.

Towards this, we study the time-dependent local Eikonal equation on the Riemannian manifold $\M$
\begin{equation} \tag{\textrm{$\mathcal{P}$}} \label{chap4_eikonal-eq}
\begin{cases}
\ddt f(x,t) + \|\grad f(x,t)\|_x = P(x), & (x,t) \in \M \setminus \G \times ]0,T[, \\
f(x,t) = f_0(x), & (x,t) \in (\G \times ]0,T[) \cup \M \times \set{0},
\end{cases}
\end{equation}
where ${\Gamma} \subset \M$ is the set of boundary points, ${P}$ is the potential function, and $f_0$ is the boundary function. $\grad f(x,t) \in T_x\M$ is the Riemannian gradient in space of $f$ at a point $x$, where $T_x\M$ is the tangent space of $\M$ at $x$ and $\|\cdot\|_x$ is the norm induced by the Riemannian metric $g$ at $x$ (see~Section~\ref{chap4_sec:def-manifold} for definitions).

%The manifold $\M$, the set of points $\widetilde{\M} \subset \M$, the boundary set $\Gamma$, the set of boundary points $\widetilde{\Gamma}$, the potential functions $P$ and $\widetilde{P}$, the boundary functions $f_0$ and $f_0^\e$,  the distance $d_{\M}(\cdot,\Gamma)$ (whose definition is given in Section~\ref{chap4_sec:def-manifold}) and the approximate distance function $\widetilde{d}$ will satisfy the following assumptions:\\
In the rest of the paper, we will work under the standing assumptions:\\
\fbox{\parbox{0.975\textwidth}{
\begin{enumerate}[label=({\textbf{H.\arabic*}})]
    \item $\M$ is a differentiable manifold of class $C^3$ which is compact, and geodesically strongly convex with $C^1$ boundary $\partial\M$. \label{chap4_assum:M}
    \item   $\widetilde{\M}$ is a finite subset of $\M$. \label{chap4_asssum:tildeM}
    \item $\G \subset \M$ and $\widetilde{\G} \subset \widetilde{\M}$ are closed sets with $\M \setminus\G$ open and $\widetilde{\M} \setminus \widetilde{\G} \subset  \M \setminus \G$.\label{chap4_assum:gamma}
    \item $P \in \LIP{\M \setminus \G}$  and $\widetilde{P} \in \LIP{\widetilde{\M} \setminus \widetilde{\G}}$ are non-negative potential functions. \label{chap4_assum:P}
    \item $f_0 \in \LIP{\M}$ and $f_0^{\e} \in \LIP{\widetilde{\M}}$. \label{chap4_assum:f_0}
    \item There exists $a_0, d_0>0$ such that $\widetilde{d}(\cdot,\Gamma)$ is $C^1$ on the neighbourhood $\NeigGama \setminus \Gamma$ where $\NeigGama \eqdef \set{x \in \M,\; \widetilde{d}(x,\Gamma) < a_0}$,  and $\norm{\grad \; \widetilde{d}(x,\Gamma)}_{x}\ge d_0$ for all $x \in \NeigGama \setminus \Gamma$.\label{chap4_assum:regulariteOmega}
    % \item There exists $a_0, d_0>0$ such that $d_{\M}(\cdot,\Gamma)$, the Riemannian distance to $\Gamma$, is $C^1$ on the neighbourhood $\NeigGama \setminus \Gamma$ where $\NeigGama \eqdef \set{x \in \M,\; d_{\M}(x,\Gamma) < a_0}$,  and $\norm{\grad \; d_{\M}(x,\Gamma)}_{x}\ge d_0$ for all $x \in \NeigGama \setminus \Gamma$.\label{chap4_assum:regulariteOmega}
    \item There exists $\delta>0$ such that for all $x \in \M$, $\cut(x) \cap \mathcal{N}_{x}^{\delta} = \emptyset$, where $\mathcal{N}_{x}^{\delta} = \{y \in \M, \, d_{\M}(x,y) \leq \delta\}$ and  $\cut(x)$ is the cut locus of $x$ defined in Section~\ref{chap4_sec:def-manifold}.\label{chap4_assum:dis-reg}
    \item There exists a constant $C_\M$ and $\xi>0$ such that   $\max_{(x,y) \in \M \times \M} |\widetilde{d}(x,y) -d_\M(x,y)| \leq C_\M \e^{1 + \xi}.$ \label{chap4_assum:tilde-d}
\end{enumerate}}}\\
Assumption \ref{chap4_assum:gamma} implies that $\partial\M \subseteq \Gamma$. In assumption \ref{chap4_asssum:tildeM}, the fact that $\widetilde\M$ is finite is quite natural since the main goal of this paper is to study Eikonal equation on graphs. However, this assumption is only required to prove the existence of a solution of the non-local problem \eqref{chap4_eikonal-eq-discrete} (see Proposition~\ref{chap4_prop:existence-J}) and is not used to show any other result. In the rest of the paper, this assumption can be replaced by the fact that $\widetilde\M$ is a compact subset of $\M$. Assumption \ref{chap4_assum:regulariteOmega} is concerned with the regularity of the distance $\widetilde{d}(.,\Gamma)$. It is used to construct super-solutions that are compatible with the boundary conditions. Furthermore, assumption \ref{chap4_assum:dis-reg} guarantees local differentiability of the squared Riemannian distance. It means that for all $y \in \mathcal{N}_{x}^\delta$, for $\delta$ sufficiently small, $y \notin \cut(x)$, and thus $d^2_{\M}(x,\cdot)$ is differentiable at $y$. Assumption \ref{chap4_assum:tilde-d}   explains how $\widetilde{d}$ approximates the Riemannian distance $d_\M$. For example, $\widetilde{d}(x,y)$ could be computed using the Euclidean distance in the $\Bar{\xi}$-offset $\Omega_\M^{\Bar{\xi}}=\{x\in \R^m: d(x,\M) \leq \Bar{\xi}\}$, where $\Bar{\xi}= \e^{2+2\xi}$. In this case, $C_\M$ depends on the dimension $m^*$. We refer to the work of Mémoli and Sapiro in \cite[Theorem 5]{memoli2005distance} in which they studied the approximation of the Riemannian distance and constructed extrinsic approximation satisfying \ref{chap4_assum:tilde-d}.  

In the context of our study, it is important to illustrate concrete examples of manifolds that satisfy the hypothesis listed above, showing how our framework can be applied to various types of manifolds. Consider first the Euclidean sphere $\mathbb{S}^n$, where $n\geq 2$. This manifold is compact and geodesically strictly convex. Being a smooth manifold of class $C^\infty$, it is also certainly of class $C^3$.
Furthermore, our assumptions are met by the compact hyperbolic space $\mathbb{H}^n$. For instance, these manifolds can be constructed in various dimensions, including 3 dimensions exemplified by the Weeks and Thurston manifolds. Moreover, we can consider the three-dimensional torus $\mathbb{T}^3$ which is a compact manifold and of class $C^\infty$. The set $\Gamma$ can be chosen as an arbitrary closed set of each manifold.  $\widetilde{M}$ can be a finite set of points, as it is always possible to select a finite subset of a compact manifold. Similarly, $\widetilde{\Gamma}$ can be chosen as a finite subset of $\widetilde{M}$, and the properties of the sets $\Gamma$ and $\widetilde{\Gamma}$ are satisfied.

%and for dimensions $n\geq 4$, as evidenced by references \cite{martelli2019compacthyperbolic,martelli2020compacthyperbolic}, where the authors provide examples of compact orientable hyperbolic manifolds that do not have any spin structure.

The length scale parameter $\e$ allows us to consider the data density. In fact, scaling $\eta$ by $\e$ aims to give significant weight to pairs of points that are far apart up to distance $\e$. In order to capture proper interactions at scale $\e$, $\eta$ has to decay to zero at an appropriate rate. Our assumptions on $\eta$ are as follows:  \\
\fbox{\parbox{0.975\textwidth}{
\begin{enumerate}[label=({\textbf{H.\arabic*}}),start=9]
    \item $\eta$ is a non-negative function. \label{chap4_etapos}
    \item $\exists \;r_{\eta}>0$ such that $\supp (\eta) \subset [0, r_\eta]$.\label{chap4_etasupp}
    \item $\exists \; a \in ]0, r_\eta[$ such that $\eta$ is decreasing on $[0,a]$ and satisfies $\eta(a)>0$. We denote by $c_\eta=\eta(a)$.\label{chap4_eta:dec}
    \item $\eta$ is $\Lip{\eta}$-lipschitz  continuous on its support.\label{chap4_eta:lip}
\end{enumerate}}}\\
These assumptions on the kernel are standard, see for example \cite{calder2019consistency,fadili2023limits}.

\subsection{Contributions}
First, we show that the local Eikonal equation \eqref{chap4_eikonal-eq} and the non-local one \eqref{chap4_eikonal-eq-discrete} are well-posed. In other terms, we prove that the solution to \eqref{chap4_eikonal-eq} and to \eqref{chap4_eikonal-eq-discrete} exist and are unique in the sense of viscosity solutions using Perron's method and the comparison principle (see Proposition \ref{chap4_prop:exis-uniq} and Proposition \ref{chap4_prop:existence-J}). Then we show regularity properties of these solutions in time and space (see Theorem \ref{chap4_lip-viscosity} and Theorem~\ref{chap4_lip-viscosity-J}).  We then derive error bounds between the solution to the local problem \eqref{chap4_eikonal-eq} and the one to the non-local problem \eqref{chap4_eikonal-eq-discrete} using the regularity properties (see Theorem \ref{chap4_thm:continuous-time-estimate}). We then use the forward Euler scheme to discretize \eqref{chap4_eikonal-eq-discrete} in time and we provide a consistency result that provides an error bound between the solution of the discretized problem \eqref{chap4_eikonal-eq-discrete-fw} and \eqref{chap4_eikonal-eq} (see Theorem \ref{chap4_thm:discrete-fw-estimate}). Finally, we apply these results to a sequence of random geometric graphs with $n$ vertices (see Theorem \ref{chap4_thm:graph-bw-estimate}). In particular, we establish that the solution to the time-discretized problem \eqref{chap4_eikonal-eq-discrete-fw} converges almost surely uniformly to the viscosity solution to the local problem \eqref{chap4_eikonal-eq} as the kernel scale parameter decreases at an appropriate rate as $n \to +\infty$ and the time step $\Delta t \to 0$, hence answering all our questions asked above. %Finally, in Section~\ref{chap4_sec:eikconvgraphs}, we apply this result to graph sequences. 

\subsection{Related work}
This work is in the continuity of our previous one \cite{fadili2023limits} where we studied limits and consistency of non-local and graph approximations of the time-dependent Eikonal equation defined on Euclidean spaces. We here extend this work to the case of smooth Riemannian manifolds, which are in fact much more in line with realistic applications.  This extension is far from trivial as it raises several difficulties and necessitates several new and careful estimates.

Though several works have considered Hamilton-Jacobi type equations on graphs whose vertices are defined on a  manifold \cite{calder2022hamilton,desquesnes2017nonmonotonic,desquesnes2013eikonal,elmoataz2008nonlocal,memoli2005distance,ta2009adaptation}, only a few of them have studied their continuum limits \cite{calder2022hamilton,memoli2005distance}. Motivated by supervised learning and data depth applications, the authors of \cite{calder2022hamilton} studied the $p$-Eikonal equation on a random geometric graph where the vertices of the graph are i.i.d random variables on an open, bounded and connected subset of $\R^m$ with a $C^{1,1}$ boundary and the kernel is smooth, non-increasing and satisfies several conditions. They prove that the continuum limit of the non-local $p$-Eikonal equation is a state-constrained Eikonal equation that recovers a geodesic density weighted distance. A theoretical and computational framework was proposed for computing intrinsic distance functions and geodesics on hypersurfaces \cite{memoli2001fast} and submanifolds \cite{memoli2005distance} given by point clouds. For this, the authors proposed to replace the intrinsic Eikonal equation on the submanifold by the corresponding extrinsic Euclidean one on an offset band and proved that this approximation is consistent. Our work goes much beyond and tackles a more general Eikonal equation defined on arbitrary geometric weighted graphs whose vertices live on a compact Riemannian submanifold. We also offer a computational framework by Forward Euler discretization in time of \eqref{chap4_eikonal-eq-discrete}. Our framework also allows to cover a much wider spectrum of applications which include the previous ones as particular cases.

%They also consider applications of the $p$-Eikonal equation to data-depth and semi-supervised learning and they use the continuum limit to prove asymptotic consistency results for both applications. 
%Moreover, the authors in \cite{el2020continuum} study continuum limits of the discretized p-Laplacian evolution problem on a special kind of graphs called the sparse graphs, with homogeneous Neumann boundary conditions. They prove well-posedness, consistency of the continuous problem, error bounds for the discrete problem and application of these results to the fully discretized problems on random graph models. 

\subsection{Outline}
The paper is structured as follows. Section~\ref{chap4_sec:def-manifold} provides prerequisites on Riemannian manifolds that are necessary to our exposition. Section~\ref{chap4_sec:exist-reg} is dedicated to establishing well-posedness in the viscosity sense of problems \eqref{chap4_eikonal-eq} and \eqref{chap4_eikonal-eq-discrete}. Section~\ref{chap4_sec:main} contains the key results of this paper. Section~\ref{chap4_subsec:main-cont} provides an error bound between the solutions to \eqref{chap4_eikonal-eq} and \eqref{chap4_eikonal-eq-discrete} in continuous time. Section~\ref{chap4_subsec:eikconvdiscrete} extends this to \eqref{chap4_eikonal-eq} and \eqref{chap4_eikonal-eq-discrete-fw}, where \eqref{chap4_eikonal-eq-discrete-fw} is a forward/explicit Euler discretization in time of \eqref{chap4_eikonal-eq-discrete}. Our results are finally specialized to the  case of geometric weighted graphs on submanifolds in Section~\ref{chap4_sec:eikconvgraphs}.

\section{Notations and prerequisites on Riemannian manifolds} \label{chap4_sec:def-manifold}

%In this section,  we recall some well-known facts about Riemannian manifolds and we provide some notations that will be used in this paper.
\subsection{Preliminaries on Riemannian manifolds}
%We begin by defining a manifold of class $C^3$ (assumption \ref{chap4_assum:M}), which is essential in order to define the differentiability of a real-valued function. Then we will define the Riemannian metric $g$ on the tangent space of every point in $\M$. The latter is used to define geometric concepts on manifolds such that lengths and distances. 
% Then we will define the differentiability of a function that is essential to define the viscosity solutions. The latter leads to define the differential of a function. Finally, we define the length of a curve in order to give the definition of the distance between two points in a manifold.
The definitions and results we are about to recall are well-known in Riemannian and differential geometry and we refer for example to \cite{lang2012fundamentals,lee2009manifolds,lee2013smooth,petersen2006riemannian} for a detailed account.

\begin{defn}[$C^3$-smooth manifold]\label{chap4_def:MclassC2}
An $m^*$-dimensional manifold $\M$ is of class $C^3$ at a point $x \in \M$ if there exists a chart $(U,\varphi)$ around $x$ such that $U$ is an open set in $\M$ containing $x$ and $ \varphi: U \to \R^{m^*} $ is a $C^3$-diffeomorphism. We say that $\M$ is of class $C^3$ if it is of class $C^3$ at each point $x\in \M$.
\end{defn}

% The following definitions describe the property of geodesically strongly convex manifold given in  \cite[Definition 2]{memoli2005distance} and the definition of a boundary of class $C^1$. These definitions are used to prove the approximation of the intrinsic distance by the Euclidean one in an offset band of $\M$.
% \begin{defn}[Geodesically strongly convex manifold]
% We say that a compact manifold $\M$ with boundary $\partial \M$ is geodesically strongly convex if for every pair of points $x$ and $y$ in $\M$, there exists a unique minimizing geodesic joining them whose interior is contained in the interior of $\M$.
% \end{defn}

% \begin{defn}[$C^1 boundary$]
% We say that the boundary of a manifold $\M$ is of class $C^1$ if    for each point $x \in \partial \M$, there exists a chart $(U, \varphi)$ around $x$, where $U$ is an open subset of $\R^m$ and $\varphi:U \to \varphi(U) \subset \M$ is a $C^1$ diffeomorphism onto its image in $\M$. 
% \end{defn}

\begin{defn}[Riemannian metric and norm]
A Riemannian metric $g$ on $\M$ is a family of inner products on tangent spaces of $\M$. In other words, for each $x \in \M$, the mapping $g_x:T_x\M \times T_x\M \To \R$, where $T_x\M$ is the tangent space of $\M$ at the point $x$, is a bilinear symmetric positive definite form denoted by $g_x(u,v) = \dotp{u}{v}_x$ for any vectors $u,v \in T_x\M$. This Riemannian inner product induces a norm $\norm{.}_x$ on $T_x\M$ defined by $\norm{v}_x=\sqrt{\dotp{v}{v}_x}$, for $v \in T_x\M$.   
\end{defn}
A (smooth) manifold whose tangent spaces are endowed with a smoothly varying inner product is called a Riemannian manifold. The smoothly varying inner product is called the Riemannian metric. Strictly speaking, a Riemannian manifold is thus a couple $(\M, g)$, where $\M$ is a manifold and $g$ is a Riemannian metric on $\M$. Nevertheless, when the Riemannian metric is unimportant, we simply talk about the Riemannian manifold $\M$.

% We recall the definition of a parallel translation (or the parallel transport) of vectors along geodesics that we will need to use. The parallel translation allows to measure the distance between two vectors  that belongs to two different tangent spaces. We can use the same method to measure distances between forms $\zeta \in T^*_xM$ and $\xi \in T^*_yM$.
% \begin{defn}[Parallel translation]
%      Let $\gamma$ be a minimizing geodesic connecting two points $x$, $y \in M$, the parallel transport from $T_xM$ to $T_yM$ along $\gamma$ is given by 
% $$\|v-\TP{yx}(w)\|_x=\|w-\TP{xy}(v)\|_y.$$
% \end{defn}

In order to define viscosity solutions, we will need to give the definition of a differentiable function on a manifold as well as its differential and gradient. 

\begin{defn}[Differentiable functions on $\M$]
Suppose that $\M$ satisfies assumption \ref{chap4_assum:M}. A real-valued function $f:\M \to \R$ is called differentiable (resp. $C^1$) at a point $x \in \M$ if there exists a chart $(U,\varphi)$ around $x$ such that 
$f\circ \varphi^{-1} : \varphi(U) \subset \R^{m^*} \to \R$ is differentiable (resp. $C^1$) at $\varphi(x)$, where $U$ is an open set in $\M$ containing $x$ and $\varphi:U \to \R^{m^*}$ is a homeomorphism. The function $f$ is differentiable (resp. $C^1$) in $\M$ if it is differentiable (resp. $C^1$) at every point in $\M$. 
\end{defn}

We would like to note here that the definition of differentiability does not depend on the choice of the chart at $x$. Indeed, given any other chart $(V,\phi)$ around $x$ where $V$ is an open set in $\M$ containing $x$ and $\phi:V \to \R^{m^*}$ is a homeomorphism, we have
$$f\circ \phi^{-1}=(f\circ \varphi^{-1})\circ(\varphi \circ \phi^{-1}): \phi(U\cap V) \To \R $$ is differentiable at $\phi(x)$ since $f\circ \varphi^{-1}$ is differentiable at $\varphi(x)$ and the transition map $\varphi \circ \phi^{-1}$ is differentiable at $\phi(x)$.

\begin{defn}[Differential and gradient of a differentiable function]
Suppose that $\M$ satisfies assumption \ref{chap4_assum:M}. Let $f$ be a scalar-valued differentiable function at $x \in \M$. The differential of $f$ at $x$ is the linear map
 \begin{align*}
    \rmD f(x) : T_x\M &\to \R \\
            v &\mapsto  \rmD f(x)[v] = \left.\frac{ d}{{ d}s} f ( \gamma (s)) \right|_{s=0} ,
 \end{align*}
where $\gamma :\; ]-1,1[\; \mapsto \M$ is a differentiable curve in $\M$ with $\gamma (0)=x$ and $\gamma'(0)=v$. The gradient of $f$ at $x$, denoted by $\grad f(x)$, is the unique element of $T_x\M$ that satisfies
\[
\dotp{\grad f(x)}{v}_x = \rmD f(x)[v], \quad \forall v \in T_x\M .
\]
\end{defn}
When $f$ depends on several parameters and the variable with respect to which the gradient is computed is not clear from the context, we specify it as a subscript of $\grad$.
 
\begin{rem}
Observe that $\rmD f(x) $ is in the cotangent space $T^*_x\M $ which is the dual space of the tangent space $T_x\M$. Moreover, for a given tangent vector $v$, the directional derivative $\rmD f(x)[v]$ is independent of the choice of the curve $\gamma$. In fact, if $\g_1$ and $\g_2$ are two curves such that $\g_1(0)=\g_2(0)=x$, and in any coordinate chart $\phi$, $\frac{d}{dt}\phi \circ \g_1 |_{t=0}=\frac{d}{dt}\phi \circ \g_2 |_{t=0}$, then by the chain rule, $f$ has the same directional derivative at $x$ along $\g_1$ and $\g_2$.
\end{rem}

The gradient of a function $f$ has the following remarkable steepest-ascent property
\[
\|\grad f(x)\|_x = \max\left\{\rmD f(x)[v]: \, v \in T_x\M,\|v\|_x \leq 1\right\} .
\]
This can also be equivalently written as
\[
\|\grad f(x)\|_x = \max\left\{\rmD f(x)[v]: \, v \in T_x\M,\|v\|_x = 1\right\} = \sup\left\{\rmD f(x)[v]: \, v \in T_x\M,\|v\|_x < 1\right\} .
\]

%\begin{defn}[Norm of the differential of a function \cite{azagra2005nonsmooth}] \label{chap4_def:norm-differential}
%Suppose that $\M$ satisfies assumption \ref{chap4_assum:M}. Let $f$ be differentiable at $x$. The norm of the differential $\nabla f(x,t) \in T^*_x\M $ at the point $x$ is defined by 
%$$\|\nabla f(x,t)\|_x=\sup\{\nabla_v f(x,t):v \in T_x\M,\|v\|_x \leq 1\}. $$
%\end{defn}

%\tcb{In this paper, we will be using the following definition of the norm of the differential of a function, which is equivalent to that of Definition \ref{chap4_def:norm-differential}, for the convenience of our proof techniques
%$$\|\nabla f(x,t)\|_x=\sup\{\nabla_v f(x,t):v \in T_x\M,\|v\|_x < 1\}. $$
%}
% \begin{defn} \label{chap4_def:cotangent}
%     The cotangent space  $T_x^*\M$ at $x \in \M$ is the dual space of the tangent space $T_x\M$. The union of all cotangent spaces of all points of $\M$ is a vector bundle $T^*\M$ called the cotangent bundle. 
% \end{defn}

\subsection{Properties of the Riemannian distance}
The metric of the Riemannian manifold $(\M,g)$ allows to define the length of a curve as follows. 
\begin{defn}[Length of a curve]
The length of a piecewise smooth curve segment $\g:[a,b] \to \M$ on a Riemannian manifold $(\M,g)$ is defined by
\begin{equation*}
L(\g)= \int_a^b \norma{\dot{\g}(s)}_{\g(s)} ds ,
\end{equation*}
where $\dot{\g}(s)$ is the velocity vector of the curve $\g$ at $s$.
\end{defn}
\begin{rem}
This length is independent of the parametrization. In other words, if $\tilde{\g}$ is any reparametrization of $\g$, i.e. $\tilde{\g}=\g \circ \varphi$, where $\varphi:[c,d] \to [a,b]$ is a diffeomorphism, then $L(\tilde{\g})=L(\g)$ (see \cite[Proposition 13.25]{lee2013smooth}).
\end{rem}

We are now ready to introduce the notion of Riemannian distance between any pair of points in $\M$. 
\begin{defn}[Riemannian distance]
The Riemannian distance between two points $x$ and $y$ in $\M$, denoted by $d_{\M}(x,y)$, is defined by
\begin{equation*}
d_{\M}(x,y) = \inf\{L(\g):  \g \text{ is a piecewise smooth curve segment on $\M$ joining $x$ and $y$} \}.
\end{equation*}
% It is easy to see that $d_\M : \M \times \M \To \R$ is a metric (in the topological sense) on $\M$ (\cite[Theorem 13.29]{lee2013smooth}).  
We define the closed Riemannian ball of center $x$ and radius $r>0$ as
\begin{equation*}
B_\M(x,r)=\{y \in \M : d_\M(x,y) \leq r\}.
\end{equation*}
\end{defn}
Since any pair of points in $\M$ can be joined by a piecewise smooth curve segment (\cite[Proposition 11.33]{lee2013smooth}), the above is well-defined. The Riemannian distance function turns $\M$ into a metric
space whose topology is the same as the given manifold topology; see \cite[Theorem~13.29]{lee2013smooth}.

Geodesics generalize the notion of straight lines on curved spaces. A geodesic $\g$ on a manifold $\M$ endowed with an affine connection is a curve with zero acceleration (i.e., constant speed).

\begin{defn}[Exponential map]
The exponential map is the mapping 
\begin{equation*}
\begin{aligned}
    \Exp_x: T_x\M &\to \M \\
    		v &\mapsto \gamma(1),
\end{aligned}
\end{equation*}
where $\gamma$ is the unique geodesic such that $\g(0)=x$ and $\dot{\g}(0)=v$. Existence and uniqueness of the geodesic is ensured whenever $v \in B(0_x,\e) \subset T_x\M$ with $\e$ small enough so that for every $t \in [0,1]$, $\Exp_x(tv) = \gamma(t)$.
\end{defn}

The regularity of the distance function on a manifold is a classical and well-understood subject. For instance, the behavior of the distance function is closely related to the structure of the notion of the cut locus of a point.
\begin{defn}[Cut locus of a point]
A point $y$ of $\M$ is in the cut locus of $x$, denoted by $\cut(x)$, if and only if there is a minimal geodesic joining $x$ to $y$ whose every extension beyond $y$ is no longer minimal. 
\end{defn}
An equivalent characterization of the cut locus can be found in \cite[Theorem~1]{wolter1979distance}.
\begin{ex}
On the Euclidean sphere, the cut locus of a point $x$ is its antipodal point. On the surface of an infinitely long cylinder, the cut locus of a point consists of the line opposite to it.
\end{ex}
%\begin{thm}[{\cite[Theorem~1]{wolter1979distance}}]
%The cut locus of $x$ on $\M$ is the closure of the set of all points in $\M$ which have at least two minimal geodesic connections to $x$.
%\end{thm}

The following proposition, proved in \cite{azagra2005nonsmooth} and used in \cite{azagra2006maximum}, provides us with some sufficient conditions for the distance function to be locally of class $C^\infty$.
\begin{prop}[Local smoothness of the distance function~{\cite[Proposition~3.9]{azagra2005nonsmooth}}]\label{chap4_prop:dist-exp}
Let $\M$ be a compact Riemannian manifold. Then there exists a constant $r>0$ such that for every $x \in \M$, the exponential map $\Exp_x$ is defined on $B(0_x,r) \subset T_x\M$ and provides a $C^\infty$ diffeomorphism 
\[
\Exp_x: B(0_x,r) \to B_\M(x,r) \eqdef \Exp_x(B(0_x,r)) .
\] 
Moreover, the distance function is given by 
\[
d_\M(x,y)=\norm{\Exp_x^{-1}(y)}_x \quad \text{for all } y \in B_\M(x,r)
\]
and for every $x \in \M$, the distance map $y \in \M \mapsto d_\M(x,y)$ is of class $C^\infty$ on $B_\M(x,r) \setminus \{x\}$.
\end{prop}

The importance of the cut locus is that the distance function on $\M$ from a point $x$ is differentiable except on the cut locus of $x$ and $x$ itself. %The distance function $d_\M$ defined on the manifold $\M$ is continuous on $\M \times \M$ but it is differentiable at $(x,y) \in \M \times \M$ if and only if there is a unique length-minimizing geodesic from $x$ to $y$. 
Moreover, we have the following useful result about its gradient given in \cite{itoh2007cut}.
\begin{lem}[Gradient of the distance function]\label{chap4_lem:grad-dist}
Let $x \in \M$. For any $y \in \M \backslash (\cut(x) \cup \{x\})$, if $\g$ is the unique minimal geodesic from $x$ to $y$, then the gradient of $d_\M(x,.)$ (with respect to the second argument) at $y$ is given by
\begin{equation*}
\grad_y d_\M(x,y)=\dot{\g}(d_\M(x,y)).
\end{equation*}
\end{lem}
The squared distance function from a point $y \in \M$ on a Riemannian manifold is smooth away from the cut locus of $y$, for a fixed point $x \in \M$ (see \cite{wolter1979distance}). 
As a consequence of the Gauss lemma (see \cite[Lemma~3.5]{do1992riemannian}), a closed form of the gradient of the squared distance function can be obtained.
\begin{lem}[Gradient of the squared distance {\cite[Lemma~4.43]{chow2007ricci}}]\label{chap4_lem:squared-distance}
Let $x \in \M$. For any $y \in \M \setminus \cut\{x\}$, the gradient of $d_\M(x,.)^2$ (with respect to the second argument) at $y$ is given by\footnote{Of course, by symmetry, the role of $x$ and $y$ can be interchanged. This is the reason we did not indicate the variable as a subscript of $\grad$.}
\begin{equation*}
    \grad_y d_\M^2(x,y)=-2\Exp^{-1}_y{(x)} \in T_y\M .
\end{equation*}
\end{lem}
\begin{rem} \label{chap4_rem:sq-dis-diff} 
Assumption~\ref{chap4_assum:dis-reg} guarantees that $d^2_{\M}(x,\cdot)$ is continuously differentiable on $\mathcal{N}^\delta_x$ with a gradient given by Lemma~\ref{chap4_lem:squared-distance}. 
\end{rem}

\subsection{Other notations}
We denote by $\abs{.}$ the Euclidean norm in $\R^{m}$, where the dimension $m$ is to be understood from the context, $\LIP{A}$ the space of Lipschitz continuous mappings on $A$ and $\Lip{h}$ the Lipschitz constant of $h \in \LIP {A}$.
Let $X$ and $Y$ be two non-empty subsets of $\M$. We define their Hausdorff distance as 
\[
\distH(X,Y) = \max\pa{\sup_{x \in X} d_\M(x,Y),\sup_{y \in Y} d_\M(y,X)} .
\]
If both $X$ and $Y$ are bounded, then $\distH(X,Y)$ is finite. Moreover, $\distH(X,Y)=0$ if and only if $X$ and $Y$ have the same closure.\\
The supremum norm on a domain $A \subset \M$ is denoted by $\normL{\cdot}_{L^\infty(A)}$.
We denote the space-time cylinders by $\M_T \eqdef \M \times [0,T]$ and $\partial \M_T \eqdef (\G \times ]0,T[) \cup \M \times \set{0}.$

\section{Well-posedness and regularity results} \label{chap4_sec:exist-reg}
In this section, we study the well-posedness in the viscosity sense as well as regularity properties of the solutions to the local Eikonal equation \eqref{chap4_eikonal-eq} (see Section \ref{chap4_subsec:exis-reg-local}) and to the non-local Eikonal equation \eqref{chap4_eikonal-eq-discrete} (see Section \ref{chap4_subsec:exist-reg-nonlocal}). Existence can be obtained by the Perron's method recalled in Theorem \ref{chap4_perron} while uniqueness and continuity are based on a comparison principle, proved in Proposition \ref{chap4_Comp-eikonale} for the local problem and in Proposition \ref{chap4_comp-discrete} for the non-local one. Our work is based on the theory of viscosity solutions which was introduced by Crandall and Lions \cite{crandall1984existence} for solving first-order Hamilton-Jacobi equations. We refer to \cite{fathi2012weak} for a general introduction to viscosity solutions on Riemannian manifolds. 
% We refer to \cite{bardi1997optimal,barles2011first} for a good introduction
\nomenclature{$\grad^+$}{superdifferential set}
\nomenclature{$\grad^-$}{subdifferential set}
\nomenclature{$|.|$}{Euclidean norm}

\subsection{Problem \eqref{chap4_eikonal-eq}}\label{chap4_subsec:exis-reg-local}
In order to define viscosity solution for problem \eqref{chap4_eikonal-eq},
we first recall the definition of upper and lower semi-continuous envelope for a locally bounded function $f:\M_T \to \R$, respectively given by 
\[
f^*(x,t) \eqdef \limsup_{(y,s)\to (x,t)} f(y,s) \qandq 
f_*(x,t) \eqdef \liminf_{(y,s)\to (x,t)} f(y,s).
\]

\begin{defn}[Viscosity solution for \eqref{chap4_eikonal-eq}] \label{chap4_def:visccauchy}
An upper semi-continuous (usc) function $f:\M \To \R$ is a viscosity sub-solution of $\eqref{chap4_eikonal-eq}$ in $\Mt$ if for every point $(x_0,t_0) \in \Mt$ and every $C^1$ function $\varphi:\mathcal{N}^h_{(x_0,t_0)} \To \R$ 
where $\mathcal{N}^h_{(x_0,t_0)} = \{(y,s) \in \Mt, \; \abs{t_0-s}\leq h \text{ and } d_\M(x_0,y)\leq h \}$ 
such that $f-\varphi$ has a maximum point at $(x_0,t_0)$, we have 
$$\ddt \varphi(x_0,t_0) + \|\grad \varphi(x_0,t_0)\|_{x_0} \leq P(x_0).$$
The function $f$ is a viscosity sub-solution of \eqref{chap4_eikonal-eq} in $\M_T$ if it satisfies moreover $f(x,t) \leq f_0(x)$ for all $(x,t) \in \partial \M_T.$

A lower semi-continuous (lsc) function $f:\M \To \R$ is a viscosity super-solution of $\eqref{chap4_eikonal-eq}$ in $\Mt$ if for every point $(x_0,t_0) \in \Mt$ and every $C^1$ function $\varphi:\mathcal{N}^h_{(x_0,t_0)} \To \R$ such that $f-\varphi$ has a minimum point at $(x_0,t_0)$, we have 
$$\ddt \varphi(x_0,t_0) + \|\grad \varphi(x_0,t_0)\|_{x_0} \geq P(x_0).$$
The function $f$ is a viscosity super-solution of \eqref{chap4_eikonal-eq} in $\M_T$ if it satisfies moreover $f(x,t) \geq f_0(x)$ for all $(x,t) \in \partial \M_T.$

Finally, a locally bounded function $f:\M_T \to \R$ is a viscosity solution of $\eqref{chap4_eikonal-eq}$ in $\M_T$ (resp. in $\Mt$) if $f^*$ is a viscosity sub-solution in $\M_T$ (resp. in $\Mt$) and $f_*$ is a viscosity super-solution of $\eqref{chap4_eikonal-eq}$ in $\M_T$ (resp. in $\Mt$).
\end{defn}

\begin{rem}\label{chap4_rem:viscsol}
    In the definition \ref{chap4_def:visccauchy}, the test function $\varphi$ can be extended to a regular function $\psi$ defined on $\M_T$ such that 
    \begin{equation*}
\psi(x,t) = 
    \begin{cases}
        \varphi(x,t) &\text{ if } (x,t) \in \mathcal{N}^h_{(x,t)}, \\
        \phi(x,t) &\text{ otherwise},
    \end{cases}
\end{equation*}
\end{rem}
where $\varphi$ is $C^1$ on $\mathcal{N}^h_{(x,t)}$. Moreover, $$\|\grad \psi(x_0,t_0)\|_{x_0} = \|\grad \varphi(x_0,t_0)\|_{x_0}.$$

We begin with a comparison principle for problem \eqref{chap4_eikonal-eq}.

\begin{prop}[Comparison principle for \eqref{chap4_eikonal-eq}] \label{chap4_Comp-eikonale}
Suppose that assumptions \ref{chap4_assum:M}-\ref{chap4_assum:f_0} and \ref{chap4_assum:dis-reg} hold. Let $f$, an usc function, be a sub-solution of \eqref{chap4_eikonal-eq} and $g$, a lsc function, be a super-solution of \eqref{chap4_eikonal-eq} in $\M_T$. Then 
$$ f\leq g \quad \text{ on } \M_T.$$
\end{prop}

The comparison principle can be proved using a variational principle as in \cite[Theorem~10]{azagra2006maximum} which is valid for quite general Hamilton-Jacobi equations. Nevertheless, their proof is non-constructive and imposes other assumptions on the manifold and the Hamiltonian. 
For instance, the manifold is supposed to be complete with positive convexity and injectivity radii and the Hamiltonian has to be an intrinsically uniformly continuous function. However, they also state in \cite[Remark~12]{azagra2006maximum} that their proof holds in other situations. Our equation is somewhat less general which allows to   give an alternative and more transparent constructive proof. 
\begin{proof}
We argue by contradiction and we suppose that there exists $(x,t) \in \M_T$ such that
    $$f(x,t)-g(x,t)>0.$$
    We consider the function $\Psi_{\tau}:(x,t) \in \M_T \longrightarrow f(x,t)-g(x,t) - \frac{\tau}{T-t}$ for $\tau$ sufficiently small and we set
$$M_\tau \eqdef \sup_{(x,t) \in \M_T} \Psi_{\tau} (x,t).$$
This function is upper semi-continuous on $\M_T$ which is compact by \ref{chap4_assum:M}, then the supremum is actually a maximum and is attained at a point denoted by $(x_\tau,t_\tau) \in \M_T$. Moreover, $\Psi_{\tau}(x_\tau,t_\tau) >0$ for $\tau>0$ sufficiently small, from the positivity assumption. In order to be able to use the definition of viscosity solutions, we duplicate the variable by considering, for $\gamma>0$, the test-function 
    $$\Psi_{\tau,\gamma}:(x,t,y,s) \in \M_T^2 \longrightarrow f(x,t)-g(y,s) - \frac{d_\M^2(x,y)}{2\gamma} -\frac{|t-s|^2}{2\gamma}-\frac{\tau}{T-t},$$
    and we set
    $$M_{\tau,\gamma} \eqdef \sup_{(x,t,y,s) \in \M_T^2} \Psi_{\tau,\gamma} (x,t,y,s).$$
    Again by upper semi-continuity and compactness, the supremum  $M_{\tau,\gamma}$ is actually a maximum attained at some point $(x_\gamma,t_\gamma,y_\gamma,s_\gamma) \in \M_T^2$. We also have for $\tau$ small enough
    $$M_{\tau,\gamma}=\Psi_{\tau,\gamma}(x_\gamma,t_\gamma,y_\gamma,s_\gamma) \geq \Psi_{\tau,\gamma}(x_\tau,t_\tau,x_\tau,t_\tau)=\Psi_{\tau}(x_\tau,t_\tau) =M_\tau> 0,$$
    which implies that
     $$ \frac{\abs{t_\g -s_\g}^2}{2\g} + \frac{d^2_{\M}(x_\gamma,y_\gamma)}{2\gamma} \leq f(x_\gamma,t_\gamma)-g(y_\gamma,s_\gamma) \leq \norm{f}_{L^\infty(\M_T)} + \norm{g}_{L^\infty(\M_T)}. $$
     We then get
     \begin{equation} \label{chap4_com-prin:dist-bdd}
         \abs{t_\g -s_\g},\;  d_\M(x_\gamma,y_\gamma) \leq c \sqrt{\gamma},
    \end{equation}
    where $c$ is a constant depending only on $\norm{f}_{L^\infty(\M_T)}$ and $\norm{g}_{L^\infty(\M_T)}$.
    Using classical arguments (see, e.g. \cite[Lemma 5.2]{barles2011first}), we deduce that there exists $(\Bar{x},\Bar{t}) \in \M \times [0,T[$ such that 
    \begin{equation}\label{chap4_prop:comp-prin-lem-barles}
        \begin{cases}
            x_\gamma,y_\gamma \longrightarrow \Bar{x}, \quad &\text{ as } \gamma \to 0, \\
            t_\gamma, s_\gamma \longrightarrow \Bar{t}, \quad &\text{ as } \gamma \to 0, \\
            \Psi_\tau(\Bar{x}, \Bar{t})= M_\tau.
        \end{cases}
    \end{equation}
    If $\Bar{t}=0$, then 
    $$0<M_\tau =f(\Bar{x},0)-g(\Bar{x},0)-\frac{\tau}{T} \leq f_0(\Bar{x})-f_0(\Bar{x})-\frac{\tau}{T}<0,$$
    leading to a contradiction. Then $\Bar{t}>0$, which implies, by \eqref{chap4_prop:comp-prin-lem-barles} that $t_\gamma,s_\g>0$, for $\gamma$ small enough. Moreover, if $\bar{x} \in \Gamma$, then 
    $$0 < M_{\tau} \leq f(\bar{x},\bar{t})-g(\bar{x},\bar{t}) = f_0(\bar{x})- f_0(\bar{x})=0,$$
    which is again absurd. Thus $x_\gamma, y_\g \in \M\setminus\Gamma$ for $\gamma$ small enough.

    The mapping $(x,t) \in \M_T \mapsto f(x,t) - \varphi_1(x,t)$, where $$\varphi_1(x,t) = g(y_\g,s_\g) + \frac{d_\M^2(x,y_\g)}{2\gamma} + \frac{|t-s_\g|^2}{2\gamma}+\frac{\tau}{T-t},$$
    is smooth on a small neighborhood $\mathcal{N}_{(x_\g,t_\g)}^\g$ (since $d^2_\M(\cdot,y_\gamma)$ is of class $C^1$, for $\g$ small enough, in view of \eqref{chap4_com-prin:dist-bdd} and \ref{chap4_assum:dis-reg}, see Remark \ref{chap4_rem:sq-dis-diff}). Since $f$ is a viscosity sub-solution of \eqref{chap4_eikonal-eq} and $f-\varphi_1$ reaches a maximum at $(x_\g,t_\g)$, we deduce that 
    $$\frac{t_\gamma - s_\gamma}{\gamma} + \frac{\tau}{T^2} \leq - \norma{\grad_x \left( \frac{d^2_\M(x_\gamma,y_\gamma)}{2\gamma}\right)}_{x_\gamma} +P(x_\gamma).$$
    Using Lemma \ref{chap4_lem:squared-distance} and Proposition \ref{chap4_prop:dist-exp}, we deduce that 
    %$$\grad_x d^2_\M(x_\gamma,y_\gamma)=-2\Exp^{-1}_{x_\gamma}(y_\gamma).$$
    %Using Proposition \ref{chap4_prop:dist-exp}, we deduce that 
    \begin{equation} \label{chap4_eq:sub-sol}
        \frac{t_\gamma - s_\gamma}{\gamma} + \frac{\tau}{T^2} \leq - \frac{d_\M(x_\gamma,y_\gamma)}{\gamma} +P(x_\gamma).
    \end{equation}
   Similarly, the mapping $(y,s) \in \M_T \mapsto g(y,s) - \varphi_2(y,s)$, where 
    $$\varphi_2(y,s) =  f(x_\g,t_\g) - \frac{d_\M^2(x_\g,y)}{2\gamma} -\frac{|t_\g-s|^2}{2\gamma}-\frac{\tau}{T-t_\g},$$
    is smooth on $\mathcal{N}_{(y_\g,s_\g)}^\g$. Since $g-\varphi_2$ attains a local minimum at $(y_\gamma,s_\gamma)$ and since $g$ is a viscosity super-solution of \eqref{chap4_eikonal-eq}, we obtain that \begin{equation}\label{chap4_eq:super-sol}
        \frac{t_\gamma-s_\gamma}{\gamma} \geq - \frac{d_\M(x_\gamma,y_\gamma)}{\gamma} + P(y_\gamma).
    \end{equation}
    Now subtracting \eqref{chap4_eq:sub-sol} and \eqref{chap4_eq:super-sol}, and then using \eqref{chap4_com-prin:dist-bdd}, we obtain
    \begin{equation*}
            \frac{\tau}{T^2} \leq P(x_\gamma) - P(y_\gamma) \leq L_P d_\M(x_\gamma,y_\gamma) \leq \Lip{P} c \sqrt{\gamma} .
    \end{equation*}
    Passing to the limit as $\gamma \to 0$ leads to a contradiction.  
\end{proof}

We now turn to the existence of the solution. In order to do that,  we will assume a compatibility property between the equation and the boundary conditions in order to construct solutions to \eqref{chap4_eikonal-eq}: \\
\fbox{\parbox{0.975\textwidth}{
\begin{enumerate}[label=({\textbf{H.\arabic*}}),start=13]
    \item There exists $\psi_b \in \LIP{\M}$, with $\psi_b(x)=f_0(x)$ for all $x \in \G$, such that $\psi_b$ is a sub-solution of $\eqref{chap4_eikonal-eq}$ in $\M_T$. \label{chap4_assum:psisssol}
\end{enumerate}}}\\
Assumption \ref{chap4_assum:psisssol} is satisfied, for instance, when $f_0=0$ and $ P \geq 0$ in \eqref{chap4_eikonal-eq}.
This setting corresponds to a time-dependent Eikonal equation whose solution can be interpreted as the minimal amount of time or distance required to travel from a point $x \in \M$ to the front $\Gamma$, where the travel speed is the inverse of $P$. 
% This example plays an important role for computing distance functions which is a key step in numerous applications including image processing or computational geometry 
For instance, it enables a fast calculation of geodesic distances using the fast marching method \cite{kimmel1998computing,sethian2000fast}.

\begin{rem}\label{chap4_rem:psib}
Assumption \ref{chap4_assum:psisssol} entails in particular that 
\[
\|\grad~ \psi_b(x)\|_x \le \norm{P}_{L^\infty(\M\setminus\Gamma)} ,
\]
and thus the Lipschitz constant $L_{\psi_b}$ satisfies 
\[
L_{\psi_b}\le \norm{P}_{L^\infty(\M\setminus\Gamma)} .
\]
\end{rem}
We then have the following result.

\begin{prop}[Existence and uniqueness for \eqref{chap4_eikonal-eq}]\label{chap4_prop:exis-uniq}
Suppose that assumptions \ref{chap4_assum:M}-\ref{chap4_assum:regulariteOmega} and \ref{chap4_assum:psisssol} hold. Then, problem \eqref{chap4_eikonal-eq} admits a unique viscosity solution $f$. Moreover, there exists a function $\bar{f} \in \LIP{\M_T}$, with a Lipschitz constant depending on $a_0$, $d_0$, $L_{f_0}$ and $\norm{P}_{L^\infty(\M \setminus \G)}$, such that 
\begin{equation}\label{chap4_eq:barrier-f}
    \psi_b \leq f \leq \bar{f} \quad \text{ on } \M_T.
\end{equation}
\end{prop}

In order to give the proof of this proposition, we need to define the notion of barrier solutions and then recall the Perron's method.
\begin{defn}[Barrier sub- and super-solution]\label{chap4_def:barriercauchy}
An usc function $\underline f :\M_T\to \R$ is a barrier sub-solution of \eqref{chap4_eikonal-eq} if it is a viscosity sub-solution
in $(\M \setminus \Gamma) \times ]0,T[$ and if it satisfies moreover
\[
\lim_{y\to x, s\to t} \underline f(y,s)=f_0(x)\quad \forall (x,t)\in \Gamma \times [0,T].
\]
A lsc function $\bar{f} :\M_T\to \R$ is a barrier super-solution of \eqref{chap4_eikonal-eq} if it is a viscosity super-solution in $(\M \setminus \Gamma) \times ]0,T[$ and if it satisfies moreover
\[
\lim_{y\to x, s\to t} \bar{f}(y,s)= f_0(x)\quad \forall (x,t) \in \Gamma \times [0,T].
\]
\end{defn}

\begin{thm}[Perron's  method]\label{chap4_perron}
Assume that there exists a barrier sub-solution $\underline f$ and a barrier super-solution $\bar{f}$ of \eqref{chap4_eikonal-eq}. Then there exists a (possibly discontinuous) viscosity solution $f$ of \eqref{chap4_eikonal-eq} satisfying moreover
\[
\underline f\le f\le \bar{f}\quad {\rm in}\; \M_T.
\]
\end{thm}

The proof of the Perron's method for Hamilton-Jacobi equations defined on manifolds can be found in \cite[Theorem 8.2]{fathi2012weak}. 

\begin{proof}[Proof of Proposition \ref{chap4_prop:exis-uniq}]
By assumption \ref{chap4_assum:psisssol}, $\psi_b$ is a barrier sub-solution of \eqref{chap4_eikonal-eq}. We then have to construct a barrier super-solution $\bar{f}$. Existence will then be a direct consequence of the Perron's method as recalled in Theorem \ref{chap4_perron} while uniqueness and continuity will be direct consequences of the comparison principle shown in Proposition \ref{chap4_Comp-eikonale}.

Let $$
\bar{f}_1(x,t) = f_0(x) + K_1 t,\qandq 
\widebar{f}_2(x,t)= f_0(x)+K_2 \widetilde{d}(x,\Gamma), \quad (x,t) \in \M_T ,
  $$
where $K_1= \norm{P}_{L^\infty(\M \backslash \G)}$, and  $K_2 > 0$ is a large enough constant to be determined later. We set 
\begin{equation}\label{chap4_eq:fmindef}
\widebar{f}(x,t) = \min(\widebar{f}_1(x,t),\widebar{f}_2(x,t))= \min(f_0(x)+K_1t,  f_0(x)+K_2\widetilde{d}(x,\Gamma)).
\end{equation}
We claim that $\widebar{f}$ is a barrier super-solution for $K_2$ well-chosen.

First, observe that 
\[
\Lip{\widebar{f}} \leq \max\bpa{\Lip{\widebar{f}_1},\Lip{\widebar{f}_2}} \leq \Lip{f_0}+\max\pa{\norm{P}_{L^\infty(\M\setminus\Gamma)},K_2},
\]
since $f_0 \in \LIP{\M}$ by \ref{chap4_assum:f_0} and $\Lip{\widetilde{d}(\cdot,\Gamma)}=1$ as $\Gamma \neq \emptyset$. In particular, $\widebar{f}$ is continuous. 

Moreover, we have for $x \in \Gamma$, 
\[
\widebar{f}_2(x,t)= f_0(x) \leq \widebar{f}_1(x,t).
\]
Hence 
\begin{equation}\label{chap4_eq:fminGamma}
\widebar{f}(x,t) =  f_0(x), \quad \forall (x,t) \in \Gamma \times [0,T],
\end{equation}
which shows, via continuity that the limit property required in Definition~\ref{chap4_def:barriercauchy} holds. It remains to prove that $\widebar{f}$ is a super-solution on $(\M\setminus \Gamma)\times ]0,T[$ for $K_2$ large enough. 

Observe first that by taking $K_2 \geq K_1T/a_0$, we have for all $x \in \M \setminus \NeigGama$ (recall that $a_0$ and $\NeigGama$ are defined in assumption \ref{chap4_assum:regulariteOmega}),
\[
\widebar{f}_2(x,t) \ge  f_0(x) + K_2 a_0 \ge f_0(x) + K_1 T  \ge \widebar{f}_1(x,t) ,
\]
and thus \eqref{chap4_eq:fmindef} becomes
\begin{equation}\label{chap4_eq:fcases}
\widebar{f}(x,t) = 
\begin{cases}
\min(\widebar{f}_1(x,t),\widebar{f}_2(x,t)) & \text{ if } (x,t) \in \NeigGama \times [0,T], \\
\widebar{f}_1(x,t) & \text{ if } (x,t) \in \M \setminus \NeigGama \times [0,T] .
\end{cases}
\end{equation}
Following Definition~\ref{chap4_def:visccauchy}, consider $\varphi \in (\M\setminus \Gamma\times ]0,T[)$ such that $\widebar{f}-\varphi$ reaches a local minimum at some point $(x_0,t_0)\in (\M \setminus \Gamma)\times ]0,T[$ and such that $\varphi \in C^1$ on a small neighborhood of $(x_0,t_0)$. This is equivalent to 
\begin{equation}\label{chap4_eq:fineqmin}
\widebar{f}(y,s)-\varphi(y,s) \geq \widebar{f}(x_0,t_0)-\varphi(x_0,t_0) ,
\end{equation}
for all $(y,s) \in (\M\setminus \Gamma)\times ]0,T[$ sufficiently close to $(x_0,t_0)$. We now distinguish two cases.

\begin{enumerate}[label={\bf Case~\arabic*}]
\item~$x_0 \in \M \setminus \NeigGama$. \label{chap4_case:f1}
In this case, since $(\M \setminus \NeigGama) \subset (\M \setminus \Gamma)$, it follows from \eqref{chap4_eq:fcases} and \eqref{chap4_eq:fineqmin} that
\[
\widebar{f}_1(y,s)-\varphi(y,s) \geq \widebar{f}(y,s)-\varphi(y,s) \geq \widebar{f}_1(x_0,t_0)-\varphi(x_0,t_0) ,
\]
for all $(y,s) \in (\M\setminus \Gamma)\times ]0,T[$ sufficiently close to $(x_0,t_0)$. As $]0,T[$ is open, we take $y=x_0$ and $s=t_0+h \in ]0,T[$ for $h>0$ sufficiently small, which gives us
\begin{equation}\label{chap4_eq:phif1ineq}
\varphi(x_0,t_0+h)-\varphi(x_0,t_0) \leq \widebar{f}_1(x_0,t_0+ h)-\widebar{f}_1(x_0,t_0)=K_1h.
\end{equation}
Dividing by $h$ and passing to the limit as $h \to 0^+$, we get 
\begin{equation}\label{chap4_eq:dphif1ubnd}
\ddt\varphi (x_0,t_0) \leq K_1 .
\end{equation}
Embarking from \eqref{chap4_eq:phif1ineq} where we replace $h$ by $-h$ yields 
\begin{equation}\label{chap4_eq:dphif1lbnd}
\ddt\varphi(x_0,t_0) \geq K_1, 
\end{equation}
and thus
\begin{equation}\label{chap4_eq:dphif1}
\ddt\varphi(x_0,t_0)= K_1 .
\end{equation}
We then deduce that
\begin{equation}\label{chap4_eq:superresineq1}
\ddt\varphi(x_0,t_0)+\norma{\grad \varphi(x_0,t_0)}-P(x_0) \ge K_1-P(x_0)\ge K_1 - \norm{P}_{L^\infty(\M\setminus\Gamma)} = 0,
\end{equation}
which shows  the desired inequality in this case.
\item~$x_0 \in \NeigGama \setminus \Gamma$. \label{chap4_case:f2}
Let $I_0 \eqdef \set{i \in \set{1,2}:~ \widebar{f}(x_0,t_0) = \widebar{f}_i(x_0,t_0)}$. Thus, for any $i_0 \in I_0$, we have from \eqref{chap4_eq:fineqmin} that
\begin{equation}\label{chap4_eq:fi0ineqmin}
\widebar{f}_{i_0}(y,s)-\varphi(y,s) \geq \widebar{f}(y,s)-\varphi(y,s) \geq \widebar{f}(x_0,t_0)-\varphi(x_0,t_0) = \widebar{f}_{i_0}(x_0,t_0)-\varphi(x_0,t_0) 
\end{equation}
for all $y \in \NeigGama \setminus \Gamma$ close enough to $x_0$. If $1 \in I_0$ we argue as in \ref{chap4_case:f1} to get a contradiction. It remains to consider the case where $I_0 = \set{2}$. Embarking from \eqref{chap4_eq:fi0ineqmin} with $i_0=2$, arguing as we have done for $\widebar{f}_1$ in \ref{chap4_case:f1} to show \eqref{chap4_eq:dphif1}, and using that $\widebar{f}_2$ is actually $t$-independent, we get in this case that 
\begin{equation}\label{chap4_eq:dphif2}
\ddt\varphi(x_0,t_0)= 0 .
\end{equation}

On the other hand, let $v \in T_{x_0}\M$ where $\norm{v}_{x_0} < 1$ and let $\gamma$ be a differentiable curve in $\M$ such that $\gamma(0)=x_0$ and $\dot{\gamma}(0)=v$. Since $\NeigGama \setminus \Gamma$ is open by \ref{chap4_assum:gamma} and \ref{chap4_assum:regulariteOmega}, we can choose $h$ small enough so that $\gamma(h) \in \NeigGama \setminus \Gamma$ and $\norm{\dot{\gamma}(s)}_{\gamma(s)}\leq 1$ for all $s\in [0,h]$. This is obtained by the continuity of $t \mapsto \norm{\dot{\gamma}(t)}$ and having that $\norm{\dot{\gamma}(0)}_{x_0} < 1$. Thus, inequality \eqref{chap4_eq:fi0ineqmin} with $i_0=2$, $s=t_0$ and $y=\gamma(h)$ becomes
\begin{equation*}
    \begin{aligned}
    \frac{(\varphi(\cdot,t_0)-K_2 \widetilde{d}(\cdot,\Gamma))(\gamma(h))-(\varphi(\cdot,t_0)-K_2 \widetilde{d}(\cdot,\Gamma))(x_0)}{h} &\leq \frac{f_0(\gamma(h))-f_0(x_0)}{h} \\
    &\leq \frac{L_{f_0} d_\M(\gamma(0),\gamma(h))}{h} \\
%    &\leq L_{f_0} \frac{L \pa{\gamma_{|[0,h]}}}{h}  \\
    &\leq \frac{1}{h} L_{f_0} \int_{0}^{h} \norm{\dot{\gamma}(s)}_{\gamma(s)} ds \\
    &\leq  L_{f_0} .
    \end{aligned}
\end{equation*}
Passing to the limit as $h \to 0^+$ and taking the maximum over $v$, we get
\[
\norm{\grad \varphi(x_0,t_0)-K_2\grad~\widetilde{d}(x_0,\Gamma)}_{x_0} \le \Lip{f_0}.
\]
Combining this inequality with \eqref{chap4_eq:dphif2} and \ref{chap4_assum:regulariteOmega}, we get
\begin{align}\label{chap4_eq:superresineq2b}
&\ddt\varphi(x_0,t_0)+\norm{\grad \varphi (x_0,t_0)}_{x_0}-P(x_0) \\
&\ge  K_2 \norm{\grad~\widetilde{d}(x_0,\Gamma)}_{x_0} - \norm{ \grad \varphi(x_0,t_0)-K_2\grad~\widetilde{d}(x_0,\Gamma)}_{x_0} - P(x_0) \nonumber\\
&\ge  K_2d_0-\Lip{f_0}-\norm{P}_{L^\infty(\M\setminus\Gamma)} \ge 0, \nonumber
\end{align}
for $K_2 \ge( \Lip{f_0}+\norm{P}_{L^\infty(\M\setminus\Gamma)})/d_0$.
\end{enumerate}

In summary, taking $K_2 \geq \max\pa{(\Lip{f_0}+\norm{P}_{L^\infty(\M\setminus\Gamma)})/d_0,K_1T/a_0}$, inequalities \eqref{chap4_eq:superresineq1} and \eqref{chap4_eq:superresineq2b} hold in each respective case, and thus the desired super-solution inequality is satisfied in all cases. We then conclude that $\widebar{f}$ is a barrier super-solution. The existence of $f$ and the bound \eqref{chap4_eq:barrier-f} are then direct consequences of Perron's method.

\end{proof}

We will now provide regularity results for the solution of \eqref{chap4_eikonal-eq}. For the space regularity, we will need to use the regularity of the Riemannian distance.

\begin{thm}[Regularity of the solution of \eqref{chap4_eikonal-eq}] \label{chap4_lip-viscosity}
Suppose that assumptions \ref{chap4_assum:M}-\ref{chap4_assum:dis-reg} and \ref{chap4_assum:psisssol} hold. Then the viscosity solution to the problem  \eqref{chap4_eikonal-eq} satisfies the following regularity properties 
\begin{equation}
    f(x,\cdot) \in \LIP{[0,T[} \qwithq \Lip{f(x,\cdot)} \leq \Lip{f_0} +  \norm{P}_{L^\infty(\M \setminus \G)}, \forall x \in \M. \label{chap4_lip-t-ct-1} 
\end{equation}
For all $x$, $y$ such that $d_\M(x,y)\leq \delta$ (with $\delta$ defined in assumption \ref{chap4_assum:dis-reg}), we have
\begin{equation}
     \abs{f(x,t)-f(y,t)}\leq K d_\M(x,y) \qwithq  K = \Lip{f_0} +  2\norm{P}_{L^\infty(\M \setminus \G)}, \forall t \in [0,T], \label{chap4_eq:lipspace} 
\end{equation}
\end{thm}

\begin{proof}
If $x \in \G$, \eqref{chap4_lip-t-ct-1} obviously holds. We then consider the case $(x,t) \in \Mt$, and we first show that for any $t\in [0,T[$,
\begin{equation}\label{chap4_liptcf}
\abs{f(x,t)-f(x,0)} \leq Lt,
\end{equation}
where $L=\norm{P}_{L^\infty(\M \setminus \G)} + L_{f_0}$. We define for $(x,t) \in \M_T$,
$$f_1(x,t)=f_0(x)-Lt \qandq f_2(x,t)=f_0(x)+Lt.$$
 We claim that $f_1$ (resp. $f_2$) is a sub-solution (resp. super-solution) of \eqref{chap4_eikonal-eq}. We have $f_1 \leq f_0$ on $\partial \M_T $. 
Now define $\varphi \in \Mt$ such that $f_1- \varphi $ reaches a local maximum at some $(x_0,t_0) \in \Mt$ and such that $\varphi \in C^1$ on a small neighborhood of $(x_0,t_0)$. This is equivalent to 
\begin{equation} \label{chap4_liptcphi}
   \varphi(x_0,t_0) - \varphi(x,t) \leq f_1(x_0,t_0) - f_1(x,t)
\end{equation}
for all $(x,t) \in \Mt$ sufficiently close to $(x_0,t_0)$. Since $]0,T[$ is open, we take $x=x_0$ and $t=t_0-h \in ]0,T[$ for $h>0$ sufficiently small, and we get 
\begin{equation}
    \varphi(x_0,t_0) - \varphi(x_0,t_0-h) \leq f_1(x_0,t_0) - f_1(x_0,t_0-h)=-Lh
\end{equation}
Dividing by $h$ and passing to the limit, we get 
\begin{equation}
    \ddt \varphi(x_0,t_0) \leq -L.
\end{equation}

% since $\M \setminus \G$ is open, we take $t=t_0$ and $x=x_0-hv_n \in \M \setminus \G$ for $h>0$ sufficiently small and for $(v_n)_n$ a sequence of tangent vectors of $x_0$ such that $\norm{v_n}_{x_0}=1$ for all $n\in \N$ in the left-hand side of \eqref{chap4_liptcphi}, we get for $n\in \N$
% \begin{equation}
%     \varphi(x_0,t_0) - \varphi(x_0-hv_n,t_0) \leq f_0(x_0) - f_0(x_0-hv_n).
% \end{equation}
% Now if we let $\g(s)$ be a differentiable curve such that $\g(0)=x_0$ and $\g'(0)=v_n$, we then have 
On the other hand, let $v \in T_{x_0}\M$ with $\norm{v}_{x_0} < 1$ and let $\gamma : ]-1,1[ \to \M$ be a differentiable curve in $\M$ such that $\gamma(0)=x_0$ and $\dot{\gamma}(0)=v$. Since $\M \setminus \Gamma$ is open, we can choose $h$ small enough so that $\gamma(-h) \in M \setminus \Gamma$ and $\norm{\dot{\gamma}(s)}_{\gamma(s)}\leq 1$ for all $s\in [-h,0]$. This is obtained thanks to the continuity of $t \mapsto \norm{\dot{\gamma}(t)}$ and having that $\norm{\dot{\gamma}(0)}_{x_0} < 1$. By \eqref{chap4_liptcphi}, taking $t=t_0$ and $x=\gamma(-h)$, we have

\begin{equation*}
    \begin{aligned}
    \frac{\varphi (\g (0),t_0) -\varphi ( \g (-h),t_0)}{h} &\leq \frac{f_0(x_0)-f_0(\gamma(-h))}{h} \\
    &\leq L_{f_0}\frac{d_{\M}(\gamma(0),\gamma(-h))}{h} \\
%    &\leq L_{f_0} \frac{L \pa{\gamma_{|[-h,0]}}}{h}  \\
    &\leq \frac{1}{h} L_{f_0} \int_{-h}^{0} \norm{\dot{\gamma}(s)}_{\gamma(s)} ds \\
    &\leq  L_{f_0} .
    \end{aligned}
\end{equation*}
% on a remplacé $\g(-h)$ par $\g(0)-h\g'(0)+ O(h)$ par développement limité.
Passing to the limit as $h \to 0$ and taking the maximum over $v$, we get 
\begin{equation*}
    \norm{\grad \varphi (x_0,t_0)}_{x_0} \leq L_{f_0}.
 \end{equation*}
We then deduce that 
\begin{equation}
\begin{aligned}
    \ddt \varphi(x_0,t_0) + \norm{\grad \varphi (x_0,t_0)}_{x_0} - P(x_0)
    &\leq -L + \Lip{f_0} - P(x_0)\\
    &\leq -L+  \Lip{f_0}+ \norm{P}_{L^\infty(\M \setminus \G)}.\\
    &=0.
\end{aligned}
\end{equation}

Therefore, this shows our claim on $f_1$. Arguing in the same way, we can prove that $f_2$ is a super-solution of \eqref{chap4_eikonal-eq}. Applying the comparison principle Proposition \ref{chap4_Comp-eikonale} twice yields that for any $(x,t) \in \M \times [0,T[$,
\begin{equation*}
    f_0(x)-Lt \leq f(x,t) \leq f_0(x)+Lt,
\end{equation*}
which shows \eqref{chap4_liptcf}.
For $h>0$ sufficiently small, we then consider the function $l(x,t)=f(x,t+h)$ for all $(x,t) \in \M_T$. Then it is easy to verify that $l$ satisfies \eqref{chap4_eikonal-eq}.
% \begin{equation*}  
% \begin{cases}
% \ddt l(x,t) + H(x,\nabla l(x,t)) = 0, & (x,t) \in (\O \setminus \G) \times ]0,T[, \\
% l(x,t) = f(x,h), & (x,t) \in (\G \times ]0,T[) \cup \O \times \set{0}.
% \end{cases}
% \end{equation*}
This implies that $f(x,t)$ and $f(x,t+h)$ are solutions of the same equation \eqref{chap4_eikonal-eq}, with initial conditions respectively $f_0(x)$ and $f(x,h)$. Applying again the comparison result (see Proposition \ref{chap4_Comp-eikonale}) and using \eqref{chap4_liptcf}, we obtain for any $(x,t)\in \M \times [0,T[$,
that
\begin{equation} \label{chap4_liptc}
\begin{aligned}
   \abs {f(x,t+h)-f(x,t)} &\leq \abs{f(x,h)-f(x,0)} \\
   &\leq Lh. 
    \end{aligned}
\end{equation}
Passing to the limit as $h \to 0$ shows the time regularity claim.\medskip
 
 We now turn to the space regularity bound \eqref{chap4_eq:lipspace} and adapt the argument of \cite[Theorem~8.2]{barles2011first}. Let $\delta' < \delta$ and $x,y \in \M$ such that $d_\M(x,y)\leq \delta'$. We introduce the test-function
 \begin{equation*}
     \Psi : (x,t,y) \in \M_T \times \M \mapsto f(x,t)-f(y,t)-K d_\M(x,y),
 \end{equation*}
  and we aim to show that this function is non-positive for every $K> \norm{P}_{L^\infty(\M\setminus\Gamma)}$ + L.
 % When $t=0$ or $(x,y) \in \G^2$, we have 
 %  $$
 %  \Psi(x,t,y)=f_0(x) -f_0(y)-Kd_\M(x,y)  \leq (\Lip{f_0}-K) d_\M(x,y),
 % $$
 %  we choose $K \geq \Lip{f_0}$ to have \eqref{chap4_eq:lipspace} holds.
 We argue by contradiction and assume that 
$$\sup_{(x,t,y) \in \M_T \times \M} \Psi(x,t,y) >0.$$
Since $\Psi$ is continuous over $\M_T \times \M$ which is compact, the supremum is actually a maximum attained at some point $(\bar{x},\bar{t},\bar{y}) \in \M_T\times \M$ with $\bar{x} \neq \bar{y}$ (otherwise $\Psi(\bar{x},\bar{t},\bar{y})=0$).
 In order to use viscosity solutions arguments, we introduce the function, for $\a>0$, 
 \begin{equation*}
     \Psi_{\a}:(x,t,y,s) \in \M_T^2 \mapsto f(x,t)-f(y,s)-Kd_\M(x,y) - \frac{\abs{ t-s}^2}{2\a}.
 \end{equation*}
  Since $f$ is continuous and $\M_T$ is compact, the supremum of $\Psi_\a$ is actually a maximum attained at some point $(x_\a,t_\a,y_\a, s_\a) \in \M_T ^2$.
%  On doit démontrer que $x_\a \to \bar{x}, $y_\a \to \bar{y} then $x_\a \neq y_\a$. t_\a, s_\a \to \bar{t}  $. 
In particular, we have \begin{equation}\label{chap4_reg_cont_max_pos}
    \Psi_\a(x_\a,t_\a,y_\a, s_\a) \geq \Psi_\a(\bar{x},\bar{t},\bar{y},\bar{t})=\Psi(\bar{x},\bar{t},\bar{y}) > 0.
\end{equation}
Observe also that for $\alpha$ sufficiently small, we cannot have $x_\a=y_\a$ as otherwise $\Psi_\a(x_\a,t_\a,y_\a, s_\a)$ would be negative, hence contradicting \eqref{chap4_reg_cont_max_pos}.

If $x_\a \in \G$, then $f(x_\a,t_\a)=f_0(x_\a)=\psi_b(x_\a)$. Moreover by Proposition \ref{chap4_prop:exis-uniq}, $\psi_b(y_\a) \leq f(y_\a,s_\a)$. It then follows that 
\begin{align*}
\Psi_\a(x_\a,t_\a,y_\a,s_\a) \le & \psi_b(x_\a)-f(y_\a,s_\a)-Kd_\M(x_\a,y_\a)\\
\le& \psi_b(x_\a)-\psi_b(y_\a)-Kd_\M(x_\a,y_\a)\\
\le& (L_{\psi_b}-K)d_\M(x_\a,y_\a)\\
\le& (\norm{P}_{L^\infty(\M\setminus\Gamma)}-K)d_\M(x_\a,y_\a) 
\end{align*}
where we used Remark~\ref{chap4_rem:psib}. Since $K \geq \norm{P}_{L^\infty(\M\setminus\Gamma)}$ we get a contradiction with the positivity of $\Psi_\a(x_\a,t_\a,y_\a,s_\a)$.

 Consider in the rest the case $x_\a\in \M \setminus \Gamma$. Classical arguments  (see, e.g., \cite[Lemma~5.2]{barles2011first}) show that $x_\alpha \to x$ and $y_\alpha \to y$. Since $d_\M(x,y)\leq \delta'$, we can assume that $d_\M(x_\alpha,y_\alpha)\leq \delta$, for $\alpha$ small enough. Therefore, \ref{chap4_assum:dis-reg} implies that $x_\alpha \notin \cut(y_\alpha)$. Moreover, $x_\a \neq y_\a$, and it follows that the function $(x,t) \mapsto f(y_\a,s_\a)+Kd_\M(x,y_\a)+\frac{|t-s_\a|^2}{2\a}$ is smooth at $(x_\a,t_\a)$. Since $f$ is a sub-solution we have
\begin{equation}\label{chap4_lip-ct-local}
\frac{t_\a-s_\a}\a + K \le P(x_\a) ,
\end{equation}
where we used Lemma~\ref{chap4_lem:grad-dist} to assert that $\norm{\grad_y d_\M(x_\a,y_\a)}_{y_\a}=1$. 
%   Let us now prove the claim. Let $v \in T_{y_\a}M$ be an arbitrary unit vector and let $\sigma:(-\delta, \delta) \to M$ such that $\sigma(0)=y_\a,$ $\sigma'(0)=v$. Now we consider $\g_s$ a unique minimizing geodesic from $x$ to $\sigma (s)$ for all $s \in (-\delta, \delta)$. We now compute using \tcb{the first variation formula}
%   \begin{equation*}
%       \langle \nabla d_x(y_\a), v \rangle =\left. \frac{d}{ds} \right|_{s=0} d_x(\sigma(s)) = \left. \frac{d}{ds} \right|_{s=0} L( \gamma_s)= \langle \g_y'(d_x(y_\a)) , v\rangle.
%   \end{equation*}
%   If $y_\a \in Cut(x) $, we have $d_\d(x,y_\a)=d(x,x+\d)+d(x + \d,y_\a)$. We choose $\d$ so that $y_\a \in \O \backslash Cut(x+\d)$and we obtain
%  \tcb{ \begin{equation*}
%   \langle\nabla d_{\d,x}(y_\a),v  \rangle = 
%   \langle\nabla d_{x}(x+\d),v \rangle + \langle\nabla d_{x+\d}(y_\a),v \rangle
%   =\langle\nabla d_{x+\d}(y_\a),v \rangle.
%   \end{equation*}}
% %   car on cherche la dérivée par rapport à $y_\a$ et $\langle\nabla d_{x}(x+\d),v \rangle$ est indépendant de $y_\a$

On the other hand, since $(x_\a,t_\a,y_\a,s_\a)$ is a maximum point of $\psi_\a$, we have for any $t \in [0,T]$
\begin{equation*}
f(x_\a,t) - \frac{\abs{ t-s_\a}^2}{2\a} \leq f(x_\a,t_\a) - \frac{\abs{ t_\a-s_\a}^2}{2\a}.
\end{equation*}
Choosing $t$ such that $t_\a-s_\a$ and $t_\a-t$ are of the same sign, and using \eqref{chap4_lip-t-ct-1}, we get
\begin{align*}
L|t-t_\a|\ge & f(x_\a,t_\a)-f(x_\a,t)\\
\ge & \frac{|t_\a-s_\a|^2}{2\a}-\frac{|t-s_\a|^2}{2\a},
\end{align*}
now using the polarization identity ($\norm{X}^2+\norm{Y}^2-2X.Y=\norm{X-Y}^2$), we obtain 
$$L|t-t_\a|\ge-\frac{|t-t_\a|^2}{2\a}+|t-t_\a|\frac{|t_\a-s_\a|}\a.$$
Dividing by $|t-t_\a|$ and taking $t\to t_\a$, we get
\[
\frac{|t_\a-s_\a|}{\a} \leq L .
\]
Injecting this estimate in \eqref{chap4_lip-ct-local}, we arrive to
\[
K \leq \norm{P}_{L^\infty(\M\setminus\Gamma)}+L .
\]
Since $K > \norm{P}_{L^\infty(\M\setminus\Gamma)}+L$, we get again a contradiction of the positivity of $\Psi_\a(x_\a,t_\a,y_\a,s_\a)$ on $\M \setminus \Gamma$. The above proof shows then that 
\[
f(x,t)-f(y,t)-Kd_\M(x,y) \leq 0
\] 
for all $(x,y,t) \in \M^2 \times [0,T]$ and every $K>2\norm{P}_{L^\infty(\M\setminus\Gamma)}+ \Lip{f_0}$, i.e., $f(\cdot,t)$ is globally Lipschitz continuous uniformly in $t$, hence providing the bound \eqref{chap4_eq:lipspace} for $d_\M(x,y) \leq \delta'$. Taking $\delta' \to \delta$ and $K \to 2\norm{P}_{L^\infty(\M\setminus\Gamma)}+ \Lip{f_0}$, we get the desired result.

 \end{proof}

\subsection{Problem \eqref{chap4_eikonal-eq-discrete}}\label{chap4_subsec:exist-reg-nonlocal}

We begin by the definition of viscosity solution to problem \eqref{chap4_eikonal-eq-discrete}

\begin{defn}[Viscosity solution for \eqref{chap4_eikonal-eq-discrete}]
An usc function $f^{\e}:\widetilde{\M}_T \To \R$ is a viscosity sub-solution to \eqref{chap4_eikonal-eq-discrete} in $\tMt$ if for every $C^1$ function $\varphi:]0,T[ \To \R$ and every point $(x_0,t_0) \in \tMt$ such that $f^\e(x_0,.)-\varphi$ has a local maximum point at $t_0 \in ]0,T[$, we have 
$$\ddt \varphi(t_0) +  \abs{\nabla ^{-}_{\eta_\e}f^\e(x_0,t_0)}_\infty \leq \widetilde{P}(x_0).$$
The function $f^\e$ is a viscosity sub-solution to \eqref{chap4_eikonal-eq-discrete} in $\widetilde{\M}_T$ if it satisfies moreover $f^\e(x,t) \leq f^\e_0(x)$ for all $(x,t) \in \partial \widetilde{\M}_T.$

A lsc function $f^\e:\widetilde{\M}_T \To \R$ is a viscosity super-solution to $\eqref{chap4_eikonal-eq-discrete}$ in $\tMt$ if for every $C^1$ function $\varphi:]0,T[ \To \R$ and every point $(x_0,t_0) \in \tMt$ such that $f^\e(x_0,.)-\varphi$ has a local minimum point at $t_0$, we have 
$$\ddt \varphi(t_0) +  \abs{\nabla ^{-}_{\eta_\e}f^\e(x_0,t_0)}_\infty \geq \widetilde{P}(x_0).$$
The function $f^\e$ is a viscosity super-solution to \eqref{chap4_eikonal-eq-discrete} in $\widetilde{\M}_T$ if it satisfies moreover $f^\e(x,t) \geq f^\e_0(x)$ for all $(x,t) \in \partial \widetilde{\M}_T.$

Finally, a locally bounded function $f^\e:\widetilde{\M}_T \to \R$ is a viscosity solution to $\eqref{chap4_eikonal-eq-discrete}$ in $\widetilde{\M}_T$ (resp. in $\tMt$) if $(f^\e)^*$ is a viscosity sub-solution and $(f^\e)_*$ is a viscosity super-solution to $\eqref{chap4_eikonal-eq-discrete}$ in $\widetilde{\M}_T$ (resp. in $\tMt$).

\end{defn}

\begin{prop}[Comparison principle for \eqref{chap4_eikonal-eq-discrete}] \label{chap4_comp-discrete}
Suppose that assumptions \ref{chap4_assum:M}-\ref{chap4_assum:gamma} and \ref{chap4_etapos} hold. Assume that $f^\e$ is a bounded viscosity sub-solution to \eqref{chap4_eikonal-eq-discrete} and $g^\e$ is a viscosity super-solution to \eqref{chap4_eikonal-eq-discrete}. Then 
$$ f^\e \leq g^\e \quad \text{ in } \widetilde{\M}_T.$$
\end{prop}

\begin{proof} 
Since problem \eqref{chap4_eikonal-eq-discrete} doesn't see the space differential of the solution, the proof is the same as the one of \cite[Theorem 2.10]{fadili2023limits} and we skip it.

\end{proof}

In the same vein as for problem \eqref{chap4_eikonal-eq}, the following assumption is intended to impose compatibility properties between \eqref{chap4_eikonal-eq-discrete} and the boundary conditions on $\partial\widetilde\M_T$:\\
\fbox{\parbox{0.975\textwidth}{
\begin{enumerate}[label=({\textbf{H.\arabic*}}),start=14]
\item There exists $\widetilde{\psi}_b\in \LIP{\M}$, with $\widetilde{\psi}_b(x)= f^\e_0(x)$ for all $x\in \widetilde{\Gamma}$, such that $\widetilde{\psi}_b$ is a sub-solution of \eqref{chap4_eikonal-eq-discrete} in $\widetilde \M_T$. \label{chap4_assum:psisssol-J}
\end{enumerate}}}\\
\begin{rem}\label{chap4_rem:psib-J}
Likewise, assumption \ref{chap4_assum:psisssol-J} entails that the Lipschitz constant $L_{\widetilde{\psi}_b}$ satisfies 
\[
L_{\widetilde{\psi}_b}\le \norm{\widetilde{P}}_{L^\infty(\widetilde{\M}\setminus\widetilde{\Gamma})} .
\]
\end{rem}
\begin{rem}
   Referring to Remark \ref{chap4_rem:psib}, we can find a discussion concerning this assumption, which is similar to the one made in the local case. Specifically, when $f_0^\e=0$ and $\Tilde{P} \geq 0$, \ref{chap4_assum:psisssol-J} holds. This example, when applied to weighted graphs (see Section \ref{chap4_sec:eikconvgraphs}), corresponds to the computation of distances on data that can be represented as a weighted graph such that point clouds, discrete images and meshes. For further details, please refer to \cite{desquesnes2017nonmonotonic,toutain2016non} and references therein. 
\end{rem}

We are now ready to provide an existence result. As for the local case, the proof is based on Perron's method and the construction of barriers.
\begin{prop}[Existence result for \eqref{chap4_eikonal-eq-discrete}]\label{chap4_prop:existence-J}
Suppose that assumptions~\ref{chap4_assum:M}--\ref{chap4_assum:f_0}, \ref{chap4_etapos}--\ref{chap4_eta:dec} and \ref{chap4_assum:psisssol-J} hold. Then, problem \eqref{chap4_eikonal-eq-discrete} admits a unique viscosity solution $f^\e$ (which is in fact continuous). Moreover, there exists a function $\widebar{f}^\e \in \LIP{\widetilde{\M}}$ such that
\begin{equation}\label{chap4_eq:barrier-f-J}
\widetilde \psi_b\le f^\e\le \widebar{f}^\e\quad {\rm in }\;\widetilde \M_T.
\end{equation}
\end{prop}

\begin{proof}
The proof is the same as the one of \cite[Proposition 2.12]{fadili2023limits} and we skip it. 

\end{proof}

\begin{thm}[Time and space regularity properties for \eqref{chap4_eikonal-eq-discrete}]\label{chap4_lip-viscosity-J}
 Suppose that assumptions \ref{chap4_assum:M}-\ref{chap4_assum:f_0}, \ref{chap4_assum:tilde-d}-\ref{chap4_eta:lip} and \ref{chap4_assum:psisssol-J} hold.
 Let $f^\e$ be the bounded continuous viscosity solution to \eqref{chap4_eikonal-eq-discrete}.
 Then 
 \begin{equation}\label{chap4_lip-t-ct}
     f^\e(x,.) \in \LIP{[0,T[} \qwithq \Lip{f^\e(x,.)}\leq L, \quad \forall x\in \M,
 \end{equation}
 where $$L=\Lip{f_0^\e}+ \norm{\widetilde{P}}_{L^\infty(\widetilde{\M} \setminus \widetilde{\G})}.$$ 
Moreover, for all $(x,y) \in \widetilde{\M}^2$ and $t\in [0,T[$ such that $\widetilde{d}(x,y)\leq a \e$, where $a$ is defined in \ref{chap4_eta:dec}, we have 
 \begin{equation} \label{chap4_sol-dis-reg-space}
     \abs{f^\e(x,t)-f^\e(y,t)} \leq \max{\pa{(a+C_\M) \norm{\widetilde{P}}_{L^\infty(\widetilde{\M} \setminus \widetilde{\G})},c_\eta^{-1} C_\eta (L+ \norm{\widetilde{P}}_{L^\infty(\widetilde{\M} \setminus \widetilde{\G})})}} \e.
 \end{equation}
 Assume also that for $(x,y) \in \widetilde{\M}^2$, there exists $k(\e) \in \N$ and a path $(x_1=x,x_2, \cdots, x_{k(\e)}=y)$ with $\widetilde{d}(x_{i},x_{i+1}) \leq a\e$, $i=1, \dots, k(\e)-1$. Then for all $t \in [0,T[$, we have 
 \begin{equation} \label{chap4_sol-dis-reg-space-global}
    \abs{f^\e(x,t)-f^\e(y,t)} \leq \max{\pa{(a+C_\M) \norm{\widetilde{P}}_{L^\infty(\widetilde{\M} \setminus \widetilde{\G})},c_\eta^{-1} C_\eta (L+ \norm{\widetilde{P}}_{L^\infty(\widetilde{\M} \setminus \widetilde{\G})})}}  k(\e)\e.
 \end{equation}
\end{thm}

\begin{proof}
 The proof of \eqref{chap4_lip-t-ct} is the same as the proof of the first part of \cite[Theorem 2.15]{fadili2023limits}.
 We begin by the proof of the space regularity estimate \eqref{chap4_sol-dis-reg-space}.
 Let $(x,t) \in \widetilde{\M}_T$. If $x\in \partial\widetilde{\M}_T$, then 
 \begin{align*}
 f^\e(x,t)-f^\e(y,t) &\leq \Tilde{\psi}_b(x)-\Tilde{\psi}_b(y) \\ &\leq \Lip{\widetilde{\psi}_b} d_\M(x,y) \\
 &\leq \norm{\widetilde{P}}_{L^\infty(\widetilde{\M} \setminus \Tilde{\Gamma})} (\Tilde{d}(x,y)+C_\M \e^{1+\xi}) \\
&\leq (a+C_\M) \e \norm{\widetilde{P}}_{L^\infty(\widetilde{\M} \setminus \Tilde{\Gamma})},
 \end{align*}
 and \eqref{chap4_sol-dis-reg-space} holds.
  Assume now that $(x,t) \in \tMt$ is such that $f^\e$ is differentiable in time at $(x,t)$.
  For such points, we have from \eqref{chap4_eikonal-eq-discrete} and \eqref{chap4_lip-t-ct} that 
  \begin{equation} \label{chap4_ineq:time reg}
  \abs{\nabla^{-}_{\eta_\e}f^\e(x,t)}_\infty \leq L+ \norm{\widetilde{P}}_{L^\infty(\widetilde{\M} \setminus \widetilde{\G})}.
  \end{equation}

  Let $y \in \widetilde{\M}$ be such that  $\Tilde{d}(x,y) \leq a\e$. We then have, recalling \ref{chap4_eta:dec} and \ref{chap4_eta:lip}, that 
  \begin{align*}
      c_\eta(\e C_\eta)^{-1} (f^\e(x,t)-f^\e(y,t)) \leq& (\e C_\eta)^{-1} \eta \left( \frac{\widetilde{d}(x,y) }{\e}\right)(f^\e(x,t)-f^\e(y,t)) \\
      \leq& \abs{\nabla^{-}_{\eta_\e}f^\e(x,t)}_\infty \\
      \leq& L+ \norm{\widetilde{P}}_{L^\infty(\widetilde{\M} \setminus \widetilde{\G})}.
  \end{align*}
  Exchanging the roles of $x$ and $y$, we get that for all $(x,y) \in \widetilde{\M}^2$ and $t\in [0,T[$ such that $f^\e(x,.)$ is differentiable in time 
  \begin{equation*}
      \abs{f^\e(x,t)-f^\e(y,t)} \leq \max{\pa{(a+C_\M) \norm{\widetilde{P}}_{L^\infty(\widetilde{\M} \setminus \widetilde{\G})},c_\eta^{-1} C_\eta (L+ \norm{\widetilde{P}}_{L^\infty(\widetilde{\M} \setminus \widetilde{\G})})}} \e.
  \end{equation*}
  If $f^\e(x,.)$ is not differentiable at $t$, then since $f^\e(x,.)$ is differentiable almost everywhere, we can deduce that there exists a sequence $(t_n)_{n \in \N}$ such that $t_n$ converges to $t$ and $f^\e(x,.)$ is differentiable at $t_n$ for all $n$. By continuity of $f^\e(x,.)$ in time, we get the result for all $(x,y) \in \widetilde{\M}^2$ and $t\in [0,T[$.\bigskip
  
 The global estimate is now a direct consequence of \eqref{chap4_sol-dis-reg-space}. Indeed, we have 
  \begin{align*}
      \abs{f^\e(x,t)-f^\e(y,t)} \leq& \sum_{i=1}^{k(\e)-1}  \abs{f^\e(x_{i+1},t)-f^\e(x_{i},t)} \\
      \leq&\max{\pa{(a+C_\M) \norm{\widetilde{P}}_{L^\infty(\widetilde{\M} \setminus \widetilde{\G})},c_\eta^{-1} C_\eta (L+ \norm{\widetilde{P}}_{L^\infty(\widetilde{\M} \setminus \widetilde{\G})})}}  \sum_{i=1}^{k(\e)-1} \e \\
      \leq& \max{\pa{(a+C_\M) \norm{\widetilde{P}}_{L^\infty(\widetilde{\M} \setminus \widetilde{\G})},c_\eta^{-1} C_\eta (L+ \norm{\widetilde{P}}_{L^\infty(\widetilde{\M} \setminus \widetilde{\G})})}}  k(\e) \e.
  \end{align*}
\end{proof}

\begin{lem} \label{chap4_lip-viscosity-J-1}
Suppose that assumptions \ref{chap4_assum:M}-\ref{chap4_assum:psisssol-J} hold. Let $f^\e$ be the bounded continuous viscosity solution to \eqref{chap4_eikonal-eq-discrete}. Assume also that 
\begin{equation} \label{chap4_eq:compatdomains}
    \max_{x\in \M} d_\M(x, \widetilde{\M}) \le a\e / 8.
\end{equation}
Then, there exists $\e_0>0$ such that for all $\e\in (0,\e_0]$ and for all $(x,y) \in \widetilde{\M}^2$ and $t\in [0,T[$, the following holds
\begin{equation} \label{chap4_eq:globlip-space-J}
    \abs{f^\e(x,t)-f^\e(y,t)} \leq K (d_\M(x,y) + \e),
\end{equation}
where $K=4 a^{-1} \max{\pa{(a+C_\M) \norm{\widetilde{P}}_{L^\infty(\widetilde{\M} \setminus \widetilde{\G})},c_\eta^{-1} C_\eta (L+ \norm{\widetilde{P}}_{L^\infty(\widetilde{\M} \setminus \widetilde{\G})})}}$.
\end{lem}

\begin{proof}
Let $(x,y) \in \widetilde{\M}^2$ and denoted by $\gamma_{xy}$ the  geodesic in $\M$ joining $x$ and $y$. We then set 
$k(\e)=\left \lceil \frac {4 d_\M(x,y)}{a\e} \right\rceil$, where $ \lceil\cdot \rceil$ denotes the ceiling. For $j\in\{0, \dots k(\e)\}$, we then define $\widetilde x_j$ such that
$$\widetilde x_j \in \gamma_{xy}\quad{\rm{and}} \quad d_\M(x,\widetilde x_j)=j\delta,$$
where $\delta =d_\M(x,y)/k(\e)\le a\e/4$. In particular $\widetilde x_0= x$ and $\widetilde x_{k(\e)}=y$.
Since $\g_{xy} \subset \M$,  the condition \eqref{chap4_eq:compatdomains} implies that for any $i\in \{1,\dots, k(\e)-1\}$, there exists $x_i \in \widetilde{\M}$ such that $d_\M(\Tilde{x}_i,x_i) \le a\e /8.$ We also set $x_0=\widetilde x_0=x$ and $x_{k(\e)}=\widetilde x_{k(\e)}=y$. We then have
\begin{equation*}
    d_\M(x_i,x_{i+1}) \leq d_\M(x_i, \Tilde{x}_i) + d_\M(\Tilde{x}_i, \Tilde{x}_{i+1}) + d_\M(\Tilde{x}_{i+1}, x_{i+1}) 
    \leq a\e/4 + \delta \le a\e/2.
\end{equation*}
In view of \ref{chap4_assum:tilde-d}, for $\e \leq \e_0$, where $\e_0= \pa{a /(2C_\M)}^{1/\xi}$, we then have that
\begin{equation*}
    \Tilde{d}(x_i,x_{i+1}) \leq d_\M(x_i,x_{i+1}) + C_\M \e^{1+\xi} \leq a\e.
\end{equation*}
This allows to infer that for any $(x,y) \in \widetilde{\M}^2$, there exists a path $(x_0=x,x_1, \cdots, x_{k(\e)}=y)$ such that $x_i \in \widetilde{\M}$ and $ \Tilde{d}(x_i,x_{i+1})\leq a\e$ for all $i$.

Injecting this in \eqref{chap4_sol-dis-reg-space-global} and using the fact that 
$$ k(\e) \leq \frac{4d_\M(x,y)}{a\e} +1,$$
we get the result.

\end{proof}

\begin{rem}\label{chap4_rem:compatdomainker}
A consequence of the proof of Lemma \ref{chap4_lip-viscosity-J-1} is that, under assumption \eqref{chap4_eq:compatdomains} and \ref{chap4_assum:tilde-d}, since $a\le r_\eta$, we have
\begin{equation}
\forall x \in \widetilde{\M}, \exists y \in \widetilde{\M}, y \neq x \text{ such that } \widetilde{d}(x,y) \in \e \supp(\eta), 
\label{chap4_assum:compatdomainker}
\end{equation}
This assumption is quite natural. It basically avoids that the non-local operator $\babs{\nabla_{\eta_\e}^- f^\e(x,t)}_\infty$ is trivially zero for all $x \in \widetilde{\M}$ when $\e$ is too small. In particular, as $\widetilde{\M}$ is finite, this condition imposes that $\widetilde{\M}$ has to fill out $\M$ at least as fast as the rate at which $\e$ goes to $0$.% In this context, assumptions \eqref{chap4_eq:compatdomains} and \ref{chap4_assum:compatdomainker} are closely related. We will elaborate more on this in Section~\ref{chap4_sec:eikconvgraphs}.
\end{rem}

\section{Consistency and error bounds}
 \label{chap4_sec:main}
 
In this section, we present some error bounds between the solution to the local problem \eqref{chap4_eikonal-eq} and the one to the non-local one \eqref{chap4_eikonal-eq-discrete}. We start with a first technical lemma which will be used in the following proofs.

\begin{lem}[Test function and space and time estimations] \label{chap4_lem:testfct-estimations}
    Suppose that assumptions \ref{chap4_assum:M}-\ref{chap4_assum:psisssol-J} hold. Let $f$ and $f^\e$ be the unique viscosity solutions to \eqref{chap4_eikonal-eq} and \eqref{chap4_eikonal-eq-discrete} respectively and consider for $\gamma>0$ and $\beta>0$, the test-function
    \begin{equation} \label{chap4_test_function_error_non-local_to_local}
\Psi_{\g,\b}(x,t,y,s)= f^\e(x,t)-f(y,s)-\frac{d_\M^2(x,y)}{2\g} -\frac{\abs{t-s}^2}{2\g}- \b t,
\end{equation} 
defined on $\widetilde\M_T \times \M_T$. Then there exists a maximum point $(\bar{x},\bar{t},\bar{y},\bar{s}) \in \widetilde{\M}_T \times \M_T$ of $\Psi_{\g,\b}$. Moreover, there exists a constant $K>0$ such that $(\bar{x},\bar{t},\bar{y},\bar{s})$ satisfies
\begin{equation} \label{chap4_lip-pro-x-ct}
d_\M(\bar{x},\bar{y}) \leq K\g.
\end{equation} 
and 
\begin{equation} \label{chap4_lip-pro-t-ct}
\abs{\bar{t}-\bar{s}} \leq K(1+\b)\g.
\end{equation}
\end{lem}

\begin{proof}
The test-function $\Psi_{\g,\b}$ is continuous  (by the continuity of $f$ and $f^\e$)  on $\widetilde{\M}_T \times \M_T$ which is compact by \ref{chap4_assum:M}-\ref{chap4_asssum:tildeM}. Hence it reaches a maximum at a point $(\bar{x},\bar{t},\bar{y},\bar{s}) \in \widetilde{\M}_T \times \M_T$. We have $\Psi_{\g,\b}(\bar{x},\bar{t},\bar{y},\bar{s}) \geq \Psi_{\g,\b}(\bar{x},\bar{t},\bar{x},\bar{s}) $ since $\bar{x} \in \widetilde{\M} \subset \M$ by \ref{chap4_asssum:tildeM}. That entails, using \eqref{chap4_eq:lipspace} (see Theorem \ref{chap4_lip-viscosity}), that
\begin{equation*}
\frac{d_\M^2(\bar{x},\bar{y})}{2\g} \leq f(\bar{x},\bar{s}) - f(\bar{y},\bar{s}) \leq K d_\M(\bar{x},\bar{y}),
\end{equation*}
therefore, we obtain that
 \begin{equation*} 
d_\M(\bar{x},\bar{y}) \leq K\g.
\end{equation*}
Similarly, we have $\Psi_{\g,\b}(\bar{x},\bar{t},\bar{y},\bar{s}) \geq \Psi_{\g,\b}(\bar{x},\bar{s},\bar{y},\bar{s})$. This implies, using \eqref{chap4_lip-t-ct} (see Theorem \ref{chap4_lip-viscosity-J}), that 
\begin{equation*}
    \frac{|\Bar{t}-\Bar{s}|^2}{2\g} - \b |\Bar{t}-\Bar{s}| \leq 
    f^\e(\Bar{x},\bar{t}) - f^\e(\Bar{x},\bar{s}) \leq K |\Bar{t}-\Bar{s}|,
\end{equation*}
thus, we get
 \begin{equation*} 
 \abs{\bar{t}-\bar{s}} \leq K(1+\b)\g.
 \end{equation*}
\end{proof}

\subsection{Continuous time non-local to local error bound} \label{chap4_subsec:main-cont}
 In this section, we provide an error estimate between viscosity solutions to problems \eqref{chap4_eikonal-eq-discrete} and \eqref{chap4_eikonal-eq}. This estimate will be instrumental in deriving the remaining error bounds. For this, we need to strengthen \eqref{chap4_eq:compatdomains} by assuming: \\ 
\fbox{\parbox{0.975\textwidth}{
 \begin{enumerate} [label=({\textbf{H.\arabic*}}),start=15]
     \item There exists $\nu > 0$ such that $\max_{x\in \M} d_\M(x,\widetilde{\M}) \leq a \e^{1+\nu} / 8$.\label{chap4_assum:compatdomains}
     % \item The inverse of the exponential map $\Exp^{-1}$ is locally an L-Lipschitz diffeomorphism. \label{chap4_assum:explip}
 \end{enumerate} }}\\
% Furthermore, \ref{chap4_assum:explip} is satisfied when $g$ is a $C^{1,1}$ pseudo-Riemannian metric, i.e. $g$ is non-degenrate \cite[Theorem 2.1]{kunzinger2014exponential}. Moreover, this assumption is satisfied when $\M$ is a Riemannian manifold of bounded geometry with geometric bounds, i.e. $\M$ is connected, complete, has uniform upper and lower bounds for the sectional curvature and a uniform lower bound for the injectivity radius \cite[Theorem 3.29]{dructu2018geometric}.
\begin{thm}[Error bound between the solutions to $\eqref{chap4_eikonal-eq}$ and $\eqref{chap4_eikonal-eq-discrete}$]\label{chap4_thm:continuous-time-estimate}
Let $T>0$ and $\e \in ]0,\e_0]$ where $\e_0=\min{(1, 1/(2r_\eta)^2, (a/(2C_\M))^{1/\xi})}$. Suppose that assumptions \ref{chap4_assum:M}-\ref{chap4_assum:compatdomains} hold, and let $f$ and $f^\e$ be respectively the unique viscosity solutions to \eqref{chap4_eikonal-eq} and \eqref{chap4_eikonal-eq-discrete}. Then, there exists a constant $K>0$, depending only on $C_\M$, $\norm{f_0}_{L^\infty(\M)}$, $\norm{P}_{L^\infty(\M\setminus\G)}$, $\Lip{f_0}$, $\Lip{f_0^\e}$, $\Lip{P}$, $\Lip{\widetilde{P}}$, $\norm{\eta}_{L^\infty}$, $\Lip{\eta}$, $C_\eta$ and $c_\eta$, such that 
\begin{equation*}
    \norm{f^\e-f}_{L^\infty(\widetilde{\M} \times [0,T[)} \leq K(T+1) \pa{\e^{\min(\nu,\xi,1/2)}+\norm{P-\widetilde{P}}_{L^\infty(\widetilde{\M}\setminus\widetilde{\Gamma})}} + \norm{f^\e_0-f_0}_{L^\infty(\widetilde{\M})} + K \distH(\Gamma,\widetilde{\Gamma}).
\end{equation*}
In particular, if $\distH(\Gamma, \widetilde{\Gamma})=O(\e^{\min(\nu,\xi,1/2)})$, then
\begin{equation*}
    \norm{f^\e-f}_{L^\infty(\widetilde{\M} \times [0,T[)} \leq K(T+1) \pa{\e^{\min(\nu,\xi,1/2)}+\norm{P-\widetilde{P}}_{L^\infty(\widetilde{\M}\setminus\widetilde{\Gamma})}} + \norm{f^\e_0-f_0}_{L^\infty(\widetilde{\M})}.
\end{equation*}
The fastest convergence rate in $\e$ is then achieved when $\nu=1/2$ and $\xi=1/2$ provided that $\distH(\M,\widetilde{\M})=O(\e^{3/2})$ and $\distH(\Gamma,\widetilde{\Gamma})=O(\e^{1/2})$.
\end{thm}

\begin{proof}
The proof is quite similar to the one of \cite[Theorem 3.2]{fadili2023limits}. The difference is that in Riemannian space, one is no longer able to use properties that are well-known in Euclidean space e.g. Cauchy-Schwarz inequality, the remarkable identities, and many other properties. The main idea is to replace the Riemannian distance between two points $x$ and $y$ in $\M$ with $\norm{\Exp^{-1}(y)}_x$ defined on the tangent space $T_x\M$. Hence, once we are on the tangent space, we can use all the properties locally. In the following, $K$ indicates a positive constant that can change from one line to the other depending on the data.

\begin{enumerate}[label=Step \arabic*.]

\item {\em Test-function and maximum point.}

According to Lemma \ref{chap4_lem:testfct-estimations}, the test function $\psi_{\g,\b}$ defined  by 
$$\Psi_{\g,\b}(x,t,y,s)= f^\e(x,t)-f(y,s)-\frac{d_\M^2(x,y)}{2\g} -\frac{\abs{t-s}^2}{2\g}- \b t,$$
attains a maximum point $(\bar{x},\bar{t},\bar{y},\bar{s}) \in \widetilde{\M}_T \times \M_T$. Moreover, we have the following estimates
\begin{equation} \label{chap4_ineq:reg dist,time}
d_\M(\bar{x},\bar{y}) \leq K\g  \qandq
\abs{\bar{t}-\bar{s}} \leq K(1+\b)\g.
\end{equation}

\item {\em Excluding interior points from the maximum.}

 We show that for $\b$ large enough, we have either $(\bar{x},\bar{t}) \in \left( \mathcal N^\a_\Gamma\cap \widetilde \M \right) \times [0,T[\ \cup\  \widetilde \M\times \{0\}$  or $(\bar{y},\bar{s}) \in \mathcal N^\a_\Gamma\times [0,T[\ \cup\  \M\times \{0\}$, for $\a=\e^{1/2}$ so that $\a >\e r_\eta$ for $\e \leq \e_0$. We argue by contradiction, assuming that $(\bar x,\bar t)\in (\widetilde \M \setminus \mathcal N^\a_\Gamma) \times ]0,T[$ and $(\bar y,\bar s)\in (\M\setminus \mathcal N^\a_\Gamma) \times ]0,T[$. For $\gamma \leq \gamma_0$, small enough, $y \mapsto d^2_\M(\Bar{x},y)$ is of class $C^1$ by assumption \ref{chap4_assum:dis-reg}. Since $(\bar{y},\bar{s})$ is a minimum point of the function $(y,s) \mapsto f(y,s)- \varphi^1(y,s)$, where 
 $$\varphi^1(y,s) =  f^\e(\bar{x},\bar{t})-\frac{d_\M^2(\bar{x},y)}{2\g} -\frac{\abs{\bar{t}-s}^2}{2\g}- \b \bar{t}$$ which is of class $C^1$, for $\gamma$ small enough, and since $f$ is a super-solution of \eqref{chap4_eikonal-eq}, we get
\begin{equation*}
    \frac{\bar{t}-\bar{s}}{\g} \geq - \frac{\norm{\grad_y d_\M^2(\Bar{x},\Bar{y})}_{\Bar{y}}}{2\g} + P(\bar{y}).
 \end{equation*}

 But according to Lemma \ref{chap4_lem:squared-distance} and Proposition \ref{chap4_prop:dist-exp}, we have that
 \begin{equation*}
     \norma{ \frac{1}{2} \grad_y d_\M^2(\Bar{x},\Bar{y})}_{\Bar{y}} = \norm{- \Exp^{-1}_{\bar{y}}(\bar{x})}_{\Bar{y}}= d_{\M} ( \bar{x},\bar{y}).
 \end{equation*}
We then obtain 
 \begin{equation} \label{chap4_sub-ineq}
 \frac{\bar{t}-\bar{s}}{\g} \geq - \frac{d_\M (\bar{x},\bar{y})}{\g}  + P(\bar{y}).
 \end{equation}
 
Similarly, since $\bar{t}$ is a maximum point of the function $t \mapsto f^\e(\Bar{x},t) - \varphi^2(t)$, where
 $$\varphi^2(t) = f(\bar{y},\bar{s})+\frac{d_\M^2(\bar{x},\Bar{y})}{2\g} +\frac{\abs{t-\Bar{s}}^2}{2\g} +\b t,$$
which is of class $C^1$ and since $f^\e$ is a viscosity sub-solution to \eqref{chap4_eikonal-eq-discrete}, we get
\begin{equation}\label{chap4_super-ineq}
    \b + \frac{\bar{t}-\bar{s}}{\g}  \leq - \abs{\nabla^{-}_{\eta_\e}f^\e(\bar{x},\bar{t})}_\infty + \widetilde{P}(\bar{x}).
    \end{equation}
On the other hand, since $(\bar{x},\bar{t},\bar{y},\bar{s})$ is a maximizer of $\Psi_{\gamma,\beta}$, we have for any $z \in \widetilde{\M}$
\begin{equation}\label{chap4_err-bd-f-d}
\Psi_{\gamma,\beta}(\bar{x},\bar{t},\bar{y},\bar{s}) - \Psi_{\gamma,\beta}(z,\bar{t},\bar{y},\bar{s}) = f^\e(\bar{x},\bar{t}) - f^\e(z,\bar{t}) - \frac {d_\M^2(\bar{x},\bar{y}) - d_\M^2(z, \bar{y})}{2\g} \geq 0 .
\end{equation}

By \ref{chap4_assum:dis-reg} and \eqref{chap4_ineq:reg dist,time}, we get that $\Bar{y} \notin Cut\{\Bar{x}\}$ and so  there exists a unique geodesic $c$ such that $c(0)=\Bar{x}$, $c(1)=\Bar{y}$ and $L(c)=d_\M(\Bar{x},\Bar{y})$. Let $0 \leq \bar{r}_\eta \leq r_\eta$ such that $C_\eta = \bar{r}_\eta \eta (\bar{r}_\eta)$.
We define $\Bar{z}=c(s)$ such that 
$d_{\M}(\Bar{x},\bar{z})=\e \bar{r}_\eta$, for $s\in [0,1]$.  In particular $d_\M(\Bar{x},\Bar{y})=d_\M(\Bar{x},\Bar{z})+d_\M(\Bar{z},\Bar{y})$. Define $\lambda \in ]0,1]$ such that $d_\M(\Bar{x},\Bar{z})=\lambda d_\M(\Bar{x},\Bar{y})$. We then have  
\begin{equation*}\label{chap4_err-bd-d(z,y)-d(x,y)}
    d_\M(\Bar{z},\Bar{y})=(1-\lambda) d_\M(\Bar{x},\Bar{y}).
\end{equation*}

We now fix $\tilde{z} \in \widetilde{\M}$ such that $d_\M(\Bar{z},\tilde{z})\leq a \e^{1+\nu} / 8$ (see \ref{chap4_assum:compatdomains}). Using  \eqref{chap4_err-bd-f-d}, we then have
\begin{equation}
    \begin{aligned}\label{chap4_err-bd-f-d-1}
        2\gamma \pa{f^\e(\bar{x},\Bar{t})-f^\e(\tilde{z},\Bar{t})}
        &\geq d_\M^2(\bar{x},\bar{y}) - d_\M^2(\tilde{z}, \bar{y}) \\
        &\geq d_\M^2(\bar{x},\bar{y}) - \pa{d_\M(\tilde{z}, \bar{z})+d_\M(\Bar{z},\Bar{y})}^2 \\
        &= d_\M^2(\bar{x},\bar{y}) - d_\M^2(\bar{z},\bar{y}) - d_\M^2(\bar{z},\tilde{z}) - 2d_\M(\bar{z},\bar{y})d_\M(\bar{z},\tilde{z})  \\
        &= d_\M^2(\bar{x},\bar{y}) - (1-\lambda)^2 d_\M^2(\bar{x},\bar{y})- d_\M^2(\bar{z},\tilde{z}) - 2d_\M(\bar{z},\bar{y})d_\M(\bar{z},\tilde{z})\\
        &= 2\lambda  d_\M^2(\bar{x},\bar{y}) - \lambda^2  d_\M^2(\bar{x},\bar{y})-d_\M^2(\bar{z},\tilde{z}) - 2d_\M(\bar{z},\bar{y})d_\M(\bar{z},\tilde{z}).
    \end{aligned}
\end{equation}

It then follows that 
 \begin{equation} \label{chap4_ineq-2-ct}
     \begin{aligned}
     |\nabla_{\eta_\e}^{-}f^\e(\bar{x},\bar{t})|_\infty
    &\geq  J_\e(\bar{x},\tilde{z}) (f^\e(\bar{x},\bar{t})-f^\e(\tilde{z},\bar{t})) \\
    &=  J_\e(\bar{x},\bar{z}) (f^\e(\bar{x},\bar{t})-f^\e(\tilde{z},\bar{t})) + \pa{J_\e(\bar{x},\tilde{z}) -J_\e(\Bar{x},\bar{z})}(f^\e(\bar{x},\bar{t})-f^\e(\tilde{z},\bar{t}))\\
    % &\geq (2\gamma)^{-1} J_\e(\bar{x},\bar{z})\pa{2\lambda  d_\M^2(\bar{x},\bar{y}) - \lambda^2  d_\M^2(\bar{x},\bar{y})-d_\M^2(\bar{z},\tilde{z}) - 2d_\M(\bar{z},\bar{y})d_\M(\bar{z},\tilde{z})} \\
    % &\quad + \pa{J_\e(\bar{x},\tilde{z}) -J_\e(\Bar{x},\bar{z})}(f^\e(\bar{x},\bar{t})-f^\e(\tilde{z},\bar{t}))\\
     &\geq \gamma^{-1}\underset{\mathrm{T}_1}{\underbrace{  \lambda  J_\e(\Bar{x},\bar{z})  d^2_\M(\Bar{x},\Bar{y}) }} -  (2\gamma)^{-1}\underset{\mathrm{T}_2}{\underbrace{ \lambda^2 d_\M^2(\bar{x},\Bar{y}) J_\e(\Bar{x},\bar{z})}} \\
     &\quad - (2\gamma)^{-1} \underset{\mathrm{T}_3}{\underbrace{ d^2_\M(\Bar{z},\tilde{z} ) J_\e(\Bar{x},\bar{z})}} - \gamma^{-1} \underset{\mathrm{T}_4}{\underbrace{ d_\M(\bar{z},\bar{y})d_\M(\bar{z},\tilde{z})J_\e( \Bar{x},\Bar{z})}}\\
     &\quad +\underset{\mathrm{T}_5}{\underbrace{\pa{J_\e(\bar{x},\tilde{z}) -J_\e(\Bar{x},\bar{z})}(f^\e(\bar{x},\bar{t})-f^\e(\tilde{z},\bar{t}))}}.
\end{aligned}
\end{equation}

We now treat each term $\mathrm{T_1},\ldots,\mathrm{T_5}$. For $\mathrm{T_1}$, using $C_\eta= \max_{0\leq t \leq r_\eta} t \eta(t)=\bar{r}_\eta \eta(\bar{r}_\eta)$, we have
\begin{equation}\label{chap4_err-bd-T1}
     \begin{aligned}
        \mathrm{T_1} &=\lambda d_\M(\Bar{x},\Bar{y}) J_\e(\Bar{x},\bar{z})  d_\M(\Bar{x},\Bar{y}) \\
        &= (\e C_\eta)^{-1} \eta \pa{\frac{\widetilde{d}(\bar{x},\bar{z})}{\e}} d_\M(\bar{x},\bar{z}) d_\M(\bar{x},\bar{y})\\
        &\geq (\e C_\eta)^{-1} \pa{ \eta \pa{\frac{d_\M(\bar{x},\bar{z})}{\e}} - C_\M\Lip{\eta} \e^{\xi} } d_\M(\bar{x},\bar{z}) d_\M(\bar{x},\bar{y}) \\
        &= C_\eta^{-1} \eta(\bar{r}_\eta)\bar{r}_\eta d_\M(\bar{x},\bar{y}) -C_\eta ^{-1} \bar{r}_\eta C_\M \Lip{\eta}  \e^{\xi} d_\M(\bar{x},\bar{y}) \\
        &= d_\M(\bar{x},\bar{y}) - K  \e^{\xi} \gamma.
    \end{aligned}
\end{equation}
Likewise, we have
\begin{equation}\label{chap4_err-bd-T2}
     \begin{aligned}
        \mathrm{T_2} &=  J_\e(\Bar{x},\bar{z}) \lambda^2 d_\M^2(\bar{x},\Bar{y}) \\
        &= (\e C_\eta)^{-1} \eta \pa{\frac{\widetilde{d}(\bar{x},\bar{z})}{\e}}d^2_\M(\bar{x},\bar{z}) \\
        &\leq (\e C_\eta)^{-1} \eta \pa{\frac{\widetilde{d}(\bar{x},\bar{z})}{\e}} \pa{\widetilde{d}(\Bar{x},\bar{z}) + C_\M \e^{1+\xi}} d_\M(\bar{x},\bar{z})\\
        &\leq  d_\M(\Bar{x},\Bar{z}) + K  \e^{\xi} d_\M(\bar{x},\bar{z})\\
        &\leq \e r_\eta + K \e^{1+\xi}\\
        &\leq K \e.
    \end{aligned}
\end{equation}
Now we turn to bound $\mathrm{T_3}$. Using the definition of $\tilde{z}$ and the fact that $\eta$ is bounded, we have
\begin{equation}\label{chap4_err-bd-T3}
    \begin{aligned}
        \mathrm{T_3} &=d^2_\M(\Bar{z},\tilde{z} ) J_\e(\Bar{x},\bar{z}) \\
        &\leq K (\e C_\eta)^{-1}\eta\pa{\frac{\widetilde{d}(\bar{x},\bar{z})}{\e}} \e^{2+2\nu} \\
        &\leq K \e^{1+2\nu}.
    \end{aligned}
\end{equation}
To bound $\mathrm{T_4}$, we have, using \eqref{chap4_ineq:reg dist,time}, the definition of $\tilde{z}$ and the fact that $\eta$ is bounded, 
\begin{equation}\label{chap4_err-bd-T4}
    \begin{aligned}
        \mathrm{T_4} &= d_\M(\bar{z},\bar{y})d_\M(\bar{z},\tilde{z})J_\e( \Bar{x},\Bar{z}) \\
        &\leq K(1-\lambda)d_\M(\bar{x},\bar{y}) \e^{1+\nu} \e^{-1 }\\
        &\leq K \gamma \e^{\nu}.
    \end{aligned}
\end{equation}
We now turn to $\mathrm{T_5}$. Using  \eqref{chap4_eq:globlip-space-J}, \eqref{chap4_ineq:reg dist,time}, \ref{chap4_eta:lip} and  \ref{chap4_assum:compatdomains}, we have
\begin{equation} \label{chap4_err-bd-T5}
    \begin{aligned}
        |\mathrm{T_5}| &= (\e C_\eta)^{-1} \left|\eta\pa{\frac{\widetilde{d}(\bar{x},\tilde{z})}{\e}}  - \eta\pa{\frac{\widetilde{d}(\bar{x},\bar{z})}{\e}} \right|. |f^\e(\bar{x},\bar{t})-f^\e(\tilde{z},\bar{t})| \\
        &\leq K \e^{-2} |\widetilde{d}(\bar{x},\tilde{z})-\widetilde{d}(\bar{x},\bar{z})| (d_\M(\bar{x},\tilde{z})+\e) \\
        &\leq K \e^{-2}  (d_\M(\bar{z},\tilde{z})+C_\M\e^{1+\xi})(d_\M(\bar{x},\bar{z})+d_\M(\bar{z},\tilde{z})+ \e) \\
        &\leq K \e^{-2} (K\e^{1+\nu}+C_\M\e^{1+\xi})(\e r_\eta + K\e^{1+\nu} + \e )\\
        &\leq K\e^{\min{(\nu,\xi)}}.
    \end{aligned}
\end{equation}
Injecting \eqref{chap4_err-bd-T1}, \eqref{chap4_err-bd-T2}, \eqref{chap4_err-bd-T3}, \eqref{chap4_err-bd-T4} and \eqref{chap4_err-bd-T5} into \eqref{chap4_ineq-2-ct}, we arrive at 
\begin{equation*}
    |\nabla_{\eta_\e}^{-}f^\e(\bar{x},\bar{t})|_\infty
    \geq \frac{d_\M(\Bar{x},\Bar{y})}{\gamma}-K\pa{\frac{\e}{\gamma} +\e^{\min{(\nu,\xi)}}}.
\end{equation*}
Injecting this bound into \eqref{chap4_super-ineq} and combining with \eqref{chap4_sub-ineq}, we deduce that if $(\bar x,\bar t)\in (\widetilde \M \setminus \mathcal N^\a_\Gamma) \times ]0,T[$ and $(\bar y,\bar s)\in (\M\setminus \mathcal N^\a_\Gamma) \times ]0,T[$, then
\begin{equation*}
\begin{aligned}
\beta 
&< 2 K \pa{\e^{\min{(\nu,\xi)}} + \frac{\e}{\gamma}} + \widetilde{P}(\bar{x}) - P(\bar{y}) \nonumber\\
&\leq K \pa{\e^{\min{(\nu,\xi)}} + \frac{\e}{\gamma}} + \Lip{P} d_\M(\bar{x},\bar{y }) + \normL{P - \widetilde{P}}_{L^\infty(\widetilde{\M} \setminus \widetilde{\Gamma})} \nonumber\\
&\leq K \pa{\e^{\min{(\nu,\xi)}} + \frac{\e}{\gamma}} + \Lip{P} K \gamma + \normL{P - \widetilde{P}}_{L^\infty(\widetilde{\M} \setminus \widetilde{\Gamma})} \nonumber\\
&\leq K \pa{\e^{\min{(\nu,\xi)}} + \frac{\e}{\gamma}+ \gamma}+\normL{P - \widetilde{P}}_{L^\infty(\widetilde{\M} \setminus \widetilde{\Gamma})} \eqdef \bar{\beta} \label{chap4_eq:etabar} ,
\end{aligned}
\end{equation*}
for large enough constant $K > 0$, where we used \ref{chap4_assum:gamma} and \ref{chap4_assum:P} in the second inequality and  estimate \eqref{chap4_ineq:reg dist,time} in the third one. Then we conclude that for $\beta \geq \bar{\beta}$ either $(\bar{x},\bar{t}) \in \mathcal N^\a_\Gamma\times [0,T[\ \cup\  \widetilde \M\times \{0\}$  or $(\bar{y},\bar{s}) \in \mathcal N^\a_\Gamma\times [0,T[\ \cup\  \M\times \{0\}$.\medskip

\item {\em Conclusion.\\} \label{chap4_step:proofconclusion}
We take $\beta \ge \bar\beta$. Assume first that $(\bar{y},\bar{s}) \in \mathcal N^\a_\Gamma\times [0,T[\ \cup\  \M\times \{0\}$. If $\bar s=0$, then 
\begin{align*}
\Psi_{{\gamma,\beta}}(\bar{x},\bar{t},\bar{y},\bar{s}) 
&\leq f^\e(\bar{x},\bar{t}) - f_0(\bar{y}) \\
&= (f^\e(\bar{x},\bar{t}) - f^\e(\bar{x},0)) + (f_0^{\e}(\bar{x}) - f_0(\bar{x})) + (f_0(\bar{x}) - f_0(\bar{y})) \\
&\leq K \bar{t} + \normL{f_0^\e-f_0}_{L^\infty(\widetilde \M)} + \Lip{f_0}d_\M(\bar{x},\bar{y})\\
&\leq K(\beta+1) \gamma + \norm{f_0^\e-f_0}_{L^\infty(\widetilde \M)},
\end{align*}
where, in the second inequality, we used \eqref{chap4_lip-t-ct} in Theorem~\ref{chap4_lip-viscosity-J} to get the first term, and \ref{chap4_assum:M} and \ref{chap4_assum:f_0} to get the last two terms. In the last inequality, we invoked \eqref{chap4_ineq:reg dist,time}. In the same way, if $\bar{y} \in \mathcal N_\Gamma^\a$ and $\bar s>0$, let $\tilde y \in \proj_{\tilde{\Gamma}}(\bar{y})$, i.e.,
\[
d_\M(\bar y ,\tilde y)= d_\M(\bar y,\widetilde \Gamma) \le \distH(\Gamma, \widetilde \Gamma)+\a .
\]

Such $\tilde y$ exists by closedness of $\widetilde{\Gamma}$, see \ref{chap4_assum:gamma}. Since \eqref{chap4_eq:compatdomains} is in force under \ref{chap4_assum:compatdomains}, \eqref{chap4_eq:globlip-space-J} holds (see Theorem~\ref{chap4_lip-viscosity-J} and Lemma~\ref{chap4_lip-viscosity-J-1}). Using this with \ref{chap4_assum:f_0} and \eqref{chap4_ineq:reg dist,time}, we obtain
\begin{equation}\label{chap4_eq:bndxonGamma}
\begin{aligned}
\Psi_{{\gamma,\beta}}(\bar{x},\bar{t},\bar{y},\bar{s}) 
&\leq f^\e(\bar{x},\bar{t}) - f_0(\bar{y}) \\
&= (f^\e(\bar{x},\bar{t}) - f^\e(\tilde{y},\bar{t})) + (f_0^\e(\tilde{y}) - f_0(\tilde{y})) + (f_0(\tilde{y}) - f_0(\bar{y})) \\
&\leq K (d_\M(\bar{x},\tilde{y}) +\e)+ \normL{f_0^\e-f_0}_{L^\infty(\widetilde \M)} + \Lip{f_0} d_\M(\tilde{y},\bar{y}) \\
&\leq K(d_\M(\bar{x},\bar{y}) +\e)+ \normL{f_0^\e-f_0}_{L^\infty(\widetilde \M)} + Kd_\M(\bar{y},\tilde{y}) + \Lip{f_0} d_\M(\bar{y},\tilde{y}) \\
&\leq K (\gamma +\e) + \norm{f_0^\e-f_0}_{L^\infty(\widetilde \M)}+ K(\distH(\Gamma,\widetilde \Gamma) +\a).
\end{aligned}
\end{equation}
We conclude that for all $(\bar{y},\bar{s}) \in \mathcal N^\a_\Gamma\times [0,T[\ \cup\  \M\times \{0\}$, and for $\beta \ge \bar{\beta}$, we have
\[
\Psi_{{\gamma,\beta}}(\bar{x},\bar{t},\bar{y},\bar{s})\leq K (\gamma +\e)  + \norm{f_0^\e-f_0}_{L^\infty(\widetilde \M)}+ K ( \distH(\Gamma,\widetilde \Gamma)+\a)+K \beta \gamma .
\]

The same bound holds for $(\bar{x},\bar{t}) \in \pa{\mathcal N^\a_\Gamma \cap \widetilde\M}\times [0,T[\ \cup\  \widetilde \M \times \{0\}$ when $\beta \geq \bar{\beta}$. Indeed, if $\bar{t}=0$ then
\begin{align*}
\Psi_{{\gamma,\beta}}(\bar{x},\bar{t},\bar{y},\bar{s}) 
&\leq f_0^\e(\bar{x}) - f(\bar{y},\bar{s}) \\
&= (f_0^\e(\bar{x}) - f_0(\bar{x})) + (f_0(\bar{x}) - f_0(\bar{y})) + (f(\bar{y},0) - f(\bar{y},\bar{s})) \\
&\leq \normL{f_0^\e - f_0}_{L^\infty(\widetilde \M)} + \Lip{f_0} d_\M(\bar{x},\bar{y}) + K\bar{s} \\
&\leq K(\beta+1) \gamma + \normL{f_0^\e - f_0}_{L^\infty(\widetilde \M)},
\end{align*}
where we have now invoked \eqref{chap4_lip-t-ct-1} in Theorem~\ref{chap4_lip-viscosity}. If $\bar{x} \in \pa{\mathcal N^\a_\Gamma \cap \widetilde\M}$ and $\bar{t} > 0$, choose $\hat{x} \in \Gamma$ in the projection of $\bar x$ on $\Gamma$. Thus, using \eqref{chap4_eq:lipspace} in Theorem~\ref{chap4_lip-viscosity}, we arrive at
\begin{align*}
\Psi_{{\gamma,\beta}}(\bar{x},\bar{t},\bar{y},\bar{s}) 
&\leq f_0^\e(\bar{x}) - f(\bar{y},\bar{s}) \\
&= (f_0^\e(\bar{x}) - f_0(\bar{x})) + (f_0(\bar{x}) - f_0(\hat{x})) + (f(\hat{x},\bar{s}) - f(\bar{y},\bar{s})) \\
&\leq \normL{f_0^\e - f_0 }_{L^\infty(\widetilde \M)} + \Lip{f_0} d_\M(\bar{x},\hat{x}) + K d_\M(\hat{x},\bar{y}) \\
&\leq  \normL{f_0^\e - f_0}_{L^\infty(\widetilde \M)} + \Lip{f_0} \a + K(\a+\gamma) \\
&\leq K (\alpha+ \gamma) + \norm{f_0^\e - f_0}_{L^\infty(\widetilde \M)}.
\end{align*}
Thus, taking $\beta = \bar{\beta}$ and $(x,t) \in \widetilde{\M}_{T}$ we have from above that,
\begin{align*}
f^\e(x,t) - f(x,t) - \bar{\beta}T \leq  &\Psi_{{\gamma,\beta}}(\bar{x},\bar{t},\bar{y},\bar{s}) \\
\leq & K (\gamma +\e)  + \norm{f_0^\e - f_0}_{L^\infty(\widetilde \M)} + K (\distH(\Gamma,\widetilde \Gamma)+\a)+ K \bar{\beta} \gamma.
\end{align*}
Before concluding, we look at what happens when we revert the role of $f$ and $f^\e$. In this case, our reasoning remains valid with only a few changes. The main ingredient is to redefine $\Psi_{{\gamma,\beta}}$ as follows
\[
\Psi_{\gamma,\beta}(x,t,y,s) = f(y,s) - f^\e(x,t) - \frac{d^2_\M(x,y)}{2\gamma} -\frac{|t-s|^2}{2\gamma} - \beta t.
\]
Then all our bounds remain true and with even  simpler arguments (see the proof of Theorem \ref{chap4_thm:discrete-fw-estimate} which used the same kind of arguments and where we revert the role of $f$ and $f^\e$). We leave the details to the reader for the sake of brevity.
Overall, we have shown that
\[ 
|f^\e(x,t) - f(x,t)| \leq  K (\gamma +\e)  + \norm{f_0^\e - f_0}_{L^\infty(\widetilde \M)} + K( \distH(\Gamma,\widetilde \Gamma)+\a)+ K \bar{\beta} (\gamma+T).
\]
With the optimal choice $\gamma = \e^{1/2}$, taking the supremum over $(x,t)$ and after rearrangement, we get (recalling that $\a = \e^{1/2}$)
\begin{align*}
\normL{f^\e-f}_{L^\infty(\widetilde{\M} \times  [0,T[)} \leq&  K\pa{(T+1)\e^{\min(\nu,\xi,1/2)} + \e} + K (T+\e^{1/2})\normL{P - \widetilde{P}}_{L^\infty(\widetilde{\M} \setminus \widetilde{\Gamma})} \\
&+ \normL{f_0^\e - f_0}_{L^\infty(\widetilde{\M})}+ K( \distH(\Gamma,\widetilde \Gamma)+ \e^{1/2}),
\end{align*}
which implies the claimed bound.

\end{enumerate}
\end{proof}

 \subsection{Forward Euler discrete time non-local to local error bound}\label{chap4_subsec:eikconvdiscrete}

In this section, we consider the time-discrete approximation of \eqref{chap4_eikonal-eq-discrete} using Forward Euler discretization. Then we show an error estimate between this approximation and the viscosity solution to \eqref{chap4_eikonal-eq}.

For a time interval $[0,T[$ and $N_T \in \N$, we use the shorthand notation $\partial\widetilde{\M}_{N_T}=(\widetilde{\Gamma}\times \{t_1,\cdots,t_{N_T}\} \cup \widetilde\M \setminus \widetilde{\Gamma}\times \{0\})$. Using the Forward/Explicit Euler discretization scheme, a time-discrete counterpart of \eqref{chap4_eikonal-eq-discrete} is given by
\begin{equation}\tag{\textrm{$\mathcal{P}_{\e}^{\rm{FD}}$}}\label{chap4_eikonal-eq-discrete-fw}
\begin{cases}
\frac{f^\e(x,t) - f^\e(x,t-\Delta t)}{\Delta t} +|\nabla^-_{\eta_\e}f^\e(x,t-\Delta t)|_\infty = \widetilde{P}(x),  & (x,t) \in   (\widetilde{\M} \setminus \widetilde{\Gamma}) \times \set{t_1,\ldots,t_{N_T}} ,\\
f^\e(x,t) = f_0^\e (x), & (x,t) \in \partial\widetilde{\M}_{N_T} ,
\end{cases}
\end{equation}
where $t_{i} = i \Delta t$ for all $i \in \set{0,\ldots,N_T}$.

We include in Appendix \ref{chap4_sec:cauchy-J-discrete} the proof of the well-posedness of 
the equation \eqref{chap4_eikonal-eq-discrete-fw}.  Indeed, Lemma~\ref{chap4_lem-existence-lip-fw} shows the existence and regularity properties in time and space of a discrete-time solution of \eqref{chap4_eikonal-eq-discrete-fw} (in the sense of Definition~\ref{chap4_def:discretesolution}). The comparison principle given in Lemma~\ref{chap4_lem-comparison} provides uniqueness.\\

We are now in a position to state the following error estimate.

\begin{thm}[Error bound between the solutions to \eqref{chap4_eikonal-eq} and \eqref{chap4_eikonal-eq-discrete-fw}]\label{chap4_thm:discrete-fw-estimate}
Let $T>0$ and $\e \in ]0,\e_0]$ where $\e_0=\min(1,1/ {(2r_\eta)^2}, (a/(2C_\M))^{1/\xi})$. Suppose that assumptions~\ref{chap4_assum:M}--\ref{chap4_assum:compatdomains} hold, and that  $\distH(\Gamma,\widetilde \Gamma) = O(\e^{\min(\nu,\xi,1/2)})$. Let $f$  be the unique viscosity solution to \eqref{chap4_eikonal-eq} and $f^\e$ be the solution to \eqref{chap4_eikonal-eq-discrete-fw}. Assume also that
\begin{eqnarray}\label{chap4_cond-delta-t}
0 < \Delta t \leq \frac{\e C_\eta}{\sup_{t \in \R_+} \eta (t)}.
\end{eqnarray} 
Then, there exists a constant $K > 0$ depending only on $C_\M$, $\norm{f_0}_{L^\infty(\M)}$, $\norm{P}_{L^\infty(\M\setminus\Gamma)}$, $\Lip{f_0}$, $\Lip{f_
0^\e}$, $\Lip{P}$, $\Lip{\widetilde{P}}$,$\norm{\eta}_{L^\infty}$, $L_{\eta}$, $C_\eta$ and $c_\eta$ such that for any $\e$ small enough
\begin{align*}
\normL{f^\e-f}_{L^\infty\pa{\widetilde{\M} \times \set{0,\ldots,t_{N_T}}}}  
\leq& K(T+1)\pa{\e^{\min(\nu,\xi,1/2)} + \Delta t^{1/2} + \frac{\Delta t}{\e}
+ \normL{P - \widetilde{P}}_{L^\infty(\widetilde{\M}\setminus\widetilde{\Gamma})}} \\
&+ \normL{f_0^\e-f_0}_{L^\infty(\widetilde{\M})}.
\end{align*}
In particular, if $\widetilde{P}=P$ on $\widetilde{\M}\setminus\widetilde{\Gamma}$ and $f_0^\e=f_0$ on $\widetilde{\M}$, then  we have
\begin{align*}
\normL{f^\e-f}_{L^\infty\pa{\widetilde{\M} \times \set{0,\ldots,t_{N_T}}}}  
\leq& K(T+1)\pa{\e^{\min(\nu,\xi,1/2)} + \Delta t^{1/2} + \frac{\Delta t}{\e}
}.
\end{align*}
The fastest convergence rate in $\e$ is then achieved when $\Delta t = O(\e^{3/2})$, $\nu=1/2$ and $\xi=1/2$ provided that $\distH(\M,\widetilde \M) = O(\e^{3/2})$ and $\distH(\Gamma,\widetilde \Gamma) = O(\e^{1/2})$. 
\end{thm}

\begin{proof}
The proof of this theorem is similar to the one of Theorem  \ref{chap4_thm:continuous-time-estimate}.  We just point out the steps where we need to correctly process the discrete time approximation. Moreover, we revert the role of $f$ and $f^\e$ to complete the details provided in the proof of Theorem~\ref{chap4_thm:continuous-time-estimate}. Therefore, we will need to use the Lipschitz regularity properties of $f^\e$ in time and space (see Lemma~\ref{chap4_lem-existence-lip-fw}). Again, $K$ will denote in this proof any positive constant that depends only on the data but may change from one line to another.

\begin{enumerate}[label=Step~\arabic*.]
\item {\em Test-function and maximum point.\\} 
For $\gamma > 0$ and $\beta > 0$, we consider maximizing over $\M_T \times \widetilde{\M}_{N_T}$ the test-function
\begin{eqnarray*}\label{chap4_varphi}
\Psi_{\gamma,\beta}(x,t,y,s)= f(x,t) - f^\e(y,s) - \frac{d^2_\M(x,y)}{2\gamma} -\frac{|t-s|^2}{2\gamma} - \beta t.
\end{eqnarray*}
 Exactly as in the proof of Lemma~\ref{chap4_lem:testfct-estimations}, the maximum is achieved at some point $(\bar{x},\bar{t}, \bar{y},\bar{t}_i) \in \M_T \times \widetilde{\M}_{N_T}$ and this maximum satisfies the following properties 
\begin{align}
d_\M(\bar{x},\bar{y}) \leq K \gamma  \qandq |\bar{t} - \bar{t}_i| \leq K(1+\beta) \gamma. \label{chap4_ineq:lip-pro}
\end{align}

\item {\em Excluding interior points from the maximum.}

\noindent We show that for $\beta$ large enough, we have either $(\bar{x},\bar{t}) \in \partial \M_T$ or $(\bar{y},\bar{t}_i) \in \partial \widetilde{\M}_{N_T}$. We argue again by contradiction and assume that $(\bar{x},\bar{t}) \in  \M\setminus \Gamma \times ]0,T[$  and $(\bar{y},\bar{t}_i) \in   \widetilde{\M}\setminus \widetilde \Gamma \times \{t_1,\dots, t_{N_T}\}$. The function $(x,t) \mapsto f(x,t) - \varphi^1(x,t)$, where 
$$\varphi^1(x,t) = f^\e(\bar{y},\Bar{t}_i) + \frac{d^2_\M(x,\bar{y})}{2\gamma} + \frac{|t-\Bar{t}_i|^2}{2\gamma} + \beta t,$$ reaches a maximum point at $(\bar{x},\bar{t})$. Using that $\varphi^1 $ is $C^1$ on a small neighborhood of $(\Bar{x}, \Bar{t})$, for $\gamma$ small enough, (see \eqref{chap4_ineq:lip-pro}, Assumption \ref{chap4_assum:dis-reg} and Remark \ref{chap4_rem:sq-dis-diff}), the fact that $f$ is a  viscosity sub-solution to \eqref{chap4_eikonal-eq} and the fact that  $\norm{\grad_{x}d^2_\M(\Bar{x},\Bar{y})}_{\bar{x}} = 2d_\M(\Bar{x},\Bar{y})$ (see Proposition \ref{chap4_prop:dist-exp} and Lemma \ref{chap4_lem:grad-dist}), we have
\begin{eqnarray}\label{chap4_1st-ineq}
\beta + \frac{\bar{t} - \bar{t}_i}{\gamma} \leq - \frac{d_\M(\bar{x},\bar{y})}{\gamma} + P(\bar{x}).
\end{eqnarray}
% We set $\varphi: (y,s) \in \widetilde{\M}_{N_T} \mapsto f(\bar x,\bar t) - \frac{d_\M(\bar{x},y)^2}{2\gamma} -\frac{|\bar{t}-s|^2}{2\gamma} - \beta \bar{t}$. In particular, $(\bar{y},\bar{t}_i)$ is the minimum point of $f^\e-\varphi$ over $\widetilde{\M}_{N_T}$. This implies that
% \begin{eqnarray*}
% f^\e(\bar{y},\bar{t}_i) - f^\e(\bar{y},\bar{t}_i- \Delta t) \leq \varphi (\bar{y},\bar{t}_i) -\varphi(\bar{y},\bar{t}_i - \Delta t),
% \end{eqnarray*}
% and so
We have $\Psi_{\g,\b}(\bar{x},\Bar{t},\Bar{y},\Bar{t_i}) \geq \Psi_{\g,\b}(\bar{x},\Bar{t},\Bar{y},\Bar{t_i}-\Delta t)$ since $(\bar{x},\bar{t},\bar{y},\bar{t}_i)$ is a maximum point of $\Psi_{\gamma,\beta}$. This implies that  
\begin{eqnarray}\label{chap4_ineq-1}
\frac{f^\e(\bar{y},\bar{t}_i) - f^\e(\bar{y},\bar{t}_i- \Delta t) }{\Delta t} \leq \frac{\bar{t} - \bar{t}_i}{\gamma} + \frac{\Delta t}{2 \gamma}.
\end{eqnarray}
Similarly, we have $\Psi_{\g,\b}(\bar{x},\Bar{t},\Bar{y},\Bar{t_i}) \geq \Psi_{\g,\b}(\bar{x},\Bar{t},z,\Bar{t_i})$, hence we get 
\[
f^\e(\bar{y},\bar{t}_i) - f^\e(z,\bar{t}_i) \leq \frac {d^2_\M(\bar{x},z) - d^2_\M(\bar{x},\bar{y})}{2\g}, \qquad \forall z \in \widetilde{\M} .
\]

On the other hand, using that $\bar{t}_i > 0$ and that $f^\e$ is a solution to \eqref{chap4_eikonal-eq-discrete-fw}, see Definition \ref{chap4_def:discretesolution}, we have 
\begin{equation}\label{chap4_2nd-ineq}
\frac{f^\e(\bar{y},\bar{t}_i) - f^\e(\bar{y},\bar{t}_i-\Delta t)}{\Delta t} = - \babs{\nabla_{\eta_\e}^- f^\e(\bar{y},\bar{t}_i-\Delta t)}_\infty + \widetilde{P}(\bar{y}).
\end{equation} 

We now estimate the right hand side of \eqref{chap4_2nd-ineq} 
\begin{equation}\label{chap4_ineq-2}
\begin{aligned}
\babs{\nabla_{\eta_\e}^- f^\e(\bar{y},\bar{t}_i-\Delta t)}_\infty
&=  \max_{z\in\widetilde{\M},\widetilde{d}(\bar{y},z) \in \e\Sg}J_\e(\bar{y},z) (f^\e(\bar{y},\bar{t}_i-\Delta t) -f^\e(z,\bar{t}_i-\Delta t)) \\
&= \max_{z\in\widetilde{\M},\widetilde{d}(\bar{y},z) \in \e\Sg}J_\e(\bar{y},z) (f^\e(\bar{y},\bar{t}_i-\Delta t) - f^\e(\bar{y},\bar{t}_i) + f^\e(\bar{y},\bar{t}_i) - f^\e(z,\bar{t}_i) \\
&\quad + f^\e(z,\bar{t}_i) - f^\e(z,\bar{t}_i-\Delta t))\\
&\leq \max_{z\in\widetilde{\M},\widetilde{d}(\bar{y},z) \in \e\Sg} J_\e(\bar{y},z) (K \Delta t + f^\e(\bar{y},\bar{t}_i) -f^\e(z,\bar{t}_i)) \\
&\leq \max_{z\in\widetilde{\M},\widetilde{d}(\bar{y},z) \in \e\Sg} J_\e(\bar{y},z)\pa{ K \Delta t +(2\gamma)^{-1}\left( d^2_\M(\bar{x},z) - d_\M^2(\bar{x},\bar{y})\right)} \\
&\leq K\Delta t \max_{z\in\widetilde{\M},\widetilde{d}(\bar{y},z) \in \e\Sg} J_\e(\bar{y},z)
 + (2\gamma)^{-1}\max_{z\in\widetilde{\M},\widetilde{d}(\bar{y},z) \in \e\Sg} J_\e(\bar{y},z) \\
 &\quad \bpa{d_\M(\bar{x},z) - d_\M(\bar{x},\bar{y})}\bpa{d_\M(\bar{x},z) - d_\M(\bar{x},\bar{y}) + 2 d_\M(\bar{x},\bar{y})}  \\
&\leq K \frac{ \Delta t}\e \sup_{t \in \R_+}\eta(t)
+ (2\gamma)^{-1}\max_{z\in\widetilde{\M},\widetilde{d}(\bar{y},z) \in \e\Sg} J_\e(\bar{y},z){d_\M(\bar{y},z)}\bpa{d_\M(\bar{y},z) + 2 d_\M(\bar{x},\bar{y})}  \\
&\leq  K\frac{\Delta t}{\e} + \frac{d_\M(\bar{x},\bar{y})}{\gamma} \max_{z\in\widetilde{\M},\widetilde{d}(\bar{y},z) \in \e\Sg} (\e C_\eta)^{-1} \pa{\widetilde{d}(\bar{y},z) + K \e^{1+\xi}}\eta\pa{\frac{\widetilde{d}(\bar{y},z)}{\e}} \\
&\quad + (2\gamma)^{-1} \max_{z\in\widetilde{\M},\widetilde{d}(\bar{y},z) \in \e\Sg} ( \e C_\eta )^{-1} \pa{\widetilde{d}(\bar{y},z) + K \e^{1+\xi}}^2 \eta \pa{\frac{\widetilde{d}(\bar{y},z)}{\e}} \\
&\leq K\frac{\Delta t}{\e} + \frac{d_\M(\bar{x},\bar{y})}{\gamma} \pa{1+K\e^{\xi}} + (2\gamma)^{-1} \max_{z\in\widetilde{\M},\widetilde{d}(\bar{y},z) \in \e\Sg}  \Bigg( \widetilde{d}(\bar{y},z) \\
&\quad + K\e^{1+\xi} (\e C_\eta)^{-1} \widetilde{d}(\bar{y},z)\eta \pa{\frac{\widetilde{d}(\bar{y},z)}{\e}} + K \e^{2+2\xi}(\e C_\eta)^{-1} \eta \pa{\frac{\widetilde{d}(\bar{y},z)}{\e}}\Bigg) \\
&\leq  K\frac{\Delta t}{\e} + \frac{d_\M(\bar{x},\bar{y})}{\gamma} +K\e^{\xi} +  (2\gamma)^{-1} \pa{K\e + K \e^{1+\xi}  +K \e^{1+2\xi}} \\
&\leq  K \left(\frac{\Delta t}{\e} +\frac{\e}{\gamma} + \e^\xi\right) + \frac{d_\M(\bar{x},\bar{y})}{\gamma}.
\end{aligned}
\end{equation}

%\begin{equation}\label{chap4_ineq-2}
%\babs{\nabla_{J_\e}^- f^\e(\bar{u},\bar{t}_i)}_\infty
%\leq  \frac{|\bar{x}-\bar{u}|}{\gamma} + K\frac{\e}{\gamma} .
%\end{equation}

Plugging \eqref{chap4_ineq-1} and \eqref{chap4_ineq-2} into \eqref{chap4_2nd-ineq} we get

\begin{eqnarray}\label{chap4_3rd-ineq}
\frac{\bar{t} - \bar{t}_i}{\gamma} + \frac{\Delta t}{2 \gamma} \geq  -K \left(\frac{\Delta t}{\e} +\frac{\e}{\gamma} +\e^{\xi} \right)-  \frac{d_\M(\bar{x},\bar{y})}{\gamma} + \widetilde{P}(\bar{y}).
\end{eqnarray}

From \eqref{chap4_1st-ineq} and \eqref{chap4_3rd-ineq}, we finally obtain

\begin{eqnarray*}
\beta &\leq& K \pa{\frac{\Delta t + \e}{ \gamma}  + \frac{\Delta t}{\e} +\e^{\xi}} + P(\bar{x}) - \widetilde{P}(\bar{y})\\
&\leq&   K \pa{\frac{\Delta t + \e}{ \gamma}  + \frac{\Delta t}{\e} +\e^{\xi}} + K d_\M(\bar{x},\bar{y}) + \normL{P - \widetilde{P}}_{L^\infty(\widetilde{\M} \setminus \widetilde{\Gamma})} \\
&<& K \pa{\frac{\Delta t + \e}{ \gamma}  + \gamma + \frac{\Delta t}{\e} +\e^{\xi} }+\normL{P - \widetilde{P}}_{L^\infty(\widetilde{\M} \setminus \widetilde{\Gamma})} \eqdef \bar{\beta}.
\end{eqnarray*}
We then conclude that either $(\bar{x},\bar{t}) \in \partial \M_T$ or $(\bar{y},\bar{t}_i) \in \partial \widetilde{\M}_{N_T}$, for $\beta\ge \bar \beta$. When reverting the roles of $f^\e$ and $f$, only $\bar{\beta}$ will be changed taking the additional term $\e^{\min({\nu,\xi})}$ (see the proof of Theorem~\ref{chap4_thm:continuous-time-estimate}). We finally use the regularity properties of $f^\e$ (see  Lemma~\ref{chap4_lem-existence-lip-fw}) and of $f$ (see Theorem \ref{chap4_lip-viscosity}) to conclude, following \ref{chap4_step:proofconclusion} in the proof of Theorem \ref{chap4_thm:continuous-time-estimate}.

% The rest of the proof is exactly the same as \ref{chap4_step:proofconclusion} in the proof of Theorem~\ref{chap4_thm:continuous-time-estimate}, where we now invoke Lemma~\ref{chap4_lem-existence-lip-fw}.
%\JF{The rest of the proof is exactly the same as \ref{chap4_step:proofconclusion} in the proof of Theorem~\ref{chap4_subsec:eikconvcont} with the optimal choice $\Delta t = \e ^{\frac{3}{2}}$.}
\end{enumerate}
\end{proof}

\section{Application to graph sequences}\label{chap4_sec:eikconvgraphs}

Let $G_n=(V_n,E_n,w_n)$ be a finite weighted graph where $V_n$ is the set of $n$ vertices $\{u_1,\cdots,u_n\}$, $E_n \subset V_n \times V_n$ is the set of edges and $w_n(u_i,u_j)$ is the weight of any edge $(u_i,u_j)$. The latter can be defined with a kernel function at scale $\e_n$ as $w_n(u_i,u_j)=(\e C_\eta)^{-1} \eta \pa{\dfrac{\widetilde{d}(u_i,u_j)}{\e}} $. 

Let $\Gamma_n \subset V_n$. For a time interval $[0,T[$ and $N_T \in \N$, we use the shorthand notation $(V_n\setminus \Gamma_n)_{N_T}=(V_n \setminus \Gamma_n) \times \set{t_1,\ldots,t_{N_T}}$ and $\partial(V_n)_{N_T}=(\Gamma_n \times \set{t_1,\ldots,t_{N_T}}) \cup V_n \times \set{0} $. We now consider the fully discretized Eikonal equation on $G_n$ with a forward Euler time-discretization as
\begin{equation}\tag{\textrm{$\mathcal{P}_{G_n}^{\rm{FD}}$}}\label{chap4_cauchy-graph-fw}
\begin{cases}
\frac{f^{n}(u,t) - f^{n}(u,t-\Delta t)}{\Delta t} + |\nabla ^{-}_{\eta_\e}f^n(u,t-\Delta t)|_\infty =  \widetilde{P}(u),  & (u,t) \in  (V_n\setminus \Gamma_n)_{N_T} ,\\
f^{n}(u,t) = f_0^n(u), & (u,t) \in \partial(V_n)_{N_T},
\end{cases}
\end{equation}
where $t_{i} = i \Delta t$ for all $i \in \set{0,\ldots,N_T}$. 

In the notation of \eqref{chap4_eikonal-eq-discrete-fw}, it is easy to identify $V_n$ with $\widetilde{\M}$ and $\Gamma_n$ with $\widetilde{\Gamma}$. We observe that by construction, $V_n$ and $\Gamma_n$ are compact sets and that $V_n \setminus \Gamma_n \subset \M \setminus \Gamma$. Our aim in this section is to establish consistency of solutions to \eqref{chap4_cauchy-graph-fw} as $n \to +\infty$ and $\Delta t \to 0$. \\

In practice, the construction of vertices $V_n$ in a graph is beyond our direct control. The specific arrangement of points may not be known, or the points can be obtained by sampling through an acquisition device (e.g., point clouds), or derived from a learning or modeling process (e.g., images). Consequently, it is more realistic to consider graphs $G_n$ on random point configurations $V_n$, and then conveniently estimate the probability of achieving a prescribed level of consistency as a function of $n$. 

To achieve this goal, we will consider a random graph model whose nodes are latent random variables independently and identically sampled on $\M$. This random graph model is inspired from \cite{bollobas2007phase} and is quite standard. More precisely, we construct $V_n$ and the boundary $\Gamma_n$ as follows:
\begin{defn}[Construction of $V_n$ and $\Gamma_n$]\label{chap4_def:randgraph}
Given a probability measure $\mu$ over $\M$ and $\e_n > 0:$ 
\begin{enumerate}
\item draw the vertices in $V_n$ as a sequence of independent and identically distributed variables $\pa{u_i}_{i=1}^n$ taking values in $\M$ and whose common distribution is $\mu$;
\item set $\Gamma_n = \set{u_i \in V_n:~ d_\M(u_i,\Gamma) \leq a\e_n^{1+\nu}/2}$, $\nu > 0$.
\end{enumerate}
\end{defn}

From now on, we assume that \\
\fbox{\parbox{0.975\textwidth}{
\begin{enumerate}[label=({\textbf{H.\arabic*}}),itemindent=5ex,start=16]
\item $\mu$ has a density $\rho$ on $\M$ with respect to the volume measure, and $\inf_{\M} \rho > 0$. \label{chap4_assum:densitylbd}
\end{enumerate}}} \\
{~}

Before stating the main result of this section, the following lemma gives a proper choice of $\e_n$ for which the construction of Definition~\ref{chap4_def:randgraph} ensures that the key assumption \ref{chap4_assum:compatdomains} is in force together with  $\Gamma_n \neq \emptyset$ and $\distH(\Gamma,\Gamma_n) = O(\e_n^{1+\nu})$ with high probability. To lighten notation, we define the event
\begin{equation}\label{chap4_eq:eventEn}
\En = \set{\text{\ref{chap4_assum:compatdomains} holds} \qandq \distH(\Gamma,\Gamma_n) \leq a\e_n^{1+\nu}/2} .
\end{equation}

Before stating the next Lemma, we need to introduce the following assumption on the radius $\d$ of the covering of $\M$   \\
\fbox{\parbox{0.975\textwidth}{
\begin{enumerate}[label=({\textbf{H.\arabic*}}),start=17]
\item  The radius of the covering of $\M$ satisfies $\delta < \min \left\{ \inj_g(\M),\frac{\pi}{\sqrt{r}},2\pi \right\}$, where $r$ is the infimum of the sectional curvature of $\M$ and $\inj_g$ is the injectivity radius of $(\M,g)$ and where we have set $\frac{\pi}{\sqrt{r}}=+\infty$ whenever $r\leq 0$. \label{chap4_assum:cover-num}
\end{enumerate}}}\\

\noindent The definitions of the sectional curvature and the injectivity radius of $\M$ are given in Appendix \ref{chap4_sec:prooflemcompatassum-graph}.
As for the examples mentioned in the introduction, they satisfy this assumption. In fact, the Euclidean sphere $\mathbb{S}^n$ possesses a constant positive sectional curvature equal to $1/R^2$, where $R$ is its radius. Therefore, the infimum of its sectional curvature $r$ is strictly positive. Thus assumption \ref{chap4_assum:cover-num} is satisfied. Furthermore, the hyperbolic manifold $\mathbb{H}^n$ has a constant negative sectional curvature, while the sectional curvature of the torus is identically zero, ensuring the validity of the condition on $\delta$ for this two examples. 
%the injectivty radius of a hyperbolic is positive and finite. It can be zero. But since the hyperbolic is of class C^3, the injectivity radius cannot be zero. The injectivity radius measures the distance at which the exponential map becomes non-injective, meaning it starts to collapse distinct points onto the same image. For a $C^3$ hyperbolic manifold, the exponential map is smooth, and thus, its injectivity radius is positive.

\begin{lem}\label{chap4_lem:compatassum-graph}
Let $V_n$ and $\Gamma_n$ generated according to Definition~\ref{chap4_def:randgraph} where $\mu$ satisfies \ref{chap4_assum:densitylbd}. Assume that $\delta$ satisfies assumption \ref{chap4_assum:cover-num}. Then, there exist two constants $K_1 > 0$ and $K_2 > 0$ that depend only on $a$, $\vol(\M)$ and $\vol(B_\M(0))$, and for any $\tau > 0$ there exists $n(\tau) \in \N$ such that for $n \geq n(\tau)$, taking

\begin{equation}\label{chap4_eq:epsn}
\e_n^{1+\nu} = K_1(1+\tau)^{1/m^*}\pa{\frac{\log n}{n}}^{1/m^*} ,
\end{equation}
the event $\En$ in \eqref{chap4_eq:eventEn} holds with probability at least $1 - K_2 n^{-\tau}$.
\end{lem}
\begin{proof}
    The proof follows a similar approach as in \cite[Appendix D]{fadili2023limits}, with the exception that the covering number of $\M$ is distinct here. We will use again compactness of $\M$ and a covering argument with a finite $\delta$-net consisting of $N(\M,\delta)$ points.
    Let $S_\delta = \set{x_1,x_2,\ldots,x_{N(\M,\delta)}}$ be a $\delta$-net of $\M$ such that for all $x \in \M$, there exists $x_j \in S_\delta$ such that $d_\M(x,x_j)\leq \delta$, i.e., $\M \subseteq \bigcup_{x_j \in S_\delta} \ball_\M(x_j,\d)$. We then have, following the same lines as in the proof of \cite[Lemma 4.2]{fadili2023limits} that
    \begin{equation*}
    \Pr\pa{\max_{x \in \M} d_\M(x,V_n) > 2\delta} 
    \leq N(\M,\delta) \pa{1-c \delta^{m^*} \vol(\ball_\M(0))}^n.
    \end{equation*}
    Using the result obtained in \cite{loubes2008kernel}, we have that 
    \begin{equation*}
        N(\M,\delta) \leq C \vol(\M) \delta^{-m^*}.
    \end{equation*}
%%C=\frac{c_{m-1}}{2m\pi^{d-1}}, where $c_{m-1}$ is the volume of the (m-1)-dimensionnal unit sphere in \R^m%%%%% 
    We therefore arrive at the bound 
    \begin{align*}
        \Pr\pa{\max_{x \in \M} d_\M(x,V_n) > 2\delta}  
        &\leq C \vol(\M) \delta^{-m^*}\pa{1-c \delta^{m^*} \vol(\ball_\M(0))}^n \\
        &\leq C \vol(\M) \delta^{-m^*} e^{-nc\delta^{m^*}\vol(\ball_M(0))}.
    \end{align*}
    Take $\delta^{m^*} = \frac{(1+\tau)}{c\vol(\ball_\M(0)) }\frac{\log n}{n}$, for any $\tau > 0$. Thus the above bound becomes
    \begin{align*}
    \Pr\pa{\max_{x \in \Omega} d(x,V_n) > 2\delta} 
    &\leq C \vol(\M) \vol(\ball_\M(0)) e^{-(1+\tau)\log n - \log (1+\tau) - \log\log n + \log n} \\
    &\leq C \vol(\M) \vol(\ball_\M(0)) e^{-\tau\log n} \\
    &= C \vol(\M) \vol(\ball_\M(0)) n^{-\tau} .
    \end{align*}
    Thus, taking 
    \[
    \e_{n}^{1+\nu} =16 a^{-1}\pa{\frac{\pa{1+\tau}}{c\vol(\ball_\M(0)) }}^{1/m^*} \pa{\frac{\log n}{n}}^{1/m^*}, 
    \]
    we have $a\e_{n}^{1+\nu}/8 \geq \delta$, and therefore \ref{chap4_assum:compatdomains} holds with probability at least $1-K_2n^{-\tau}$.
    The proof of the estimation of the probability of the event
$\set{\distH(\Gamma,\Gamma_n) \leq a\e_n^{1+\nu}/2}$ aligns with that presented in \cite[Appendix D]{fadili2023limits}, and for brevity, we skip it.
\end{proof}

\noindent
We are now ready to establish a quantified version of uniform convergence in probability of $f^n$ towards $f$.

\begin{thm}[]\label{chap4_thm:graph-bw-estimate}
Let $T$, $\nu >0$, and $V_n$ and $\Gamma_n$ be constructed according to Definition~\ref{chap4_def:randgraph} where $\mu$ satisfies \ref{chap4_assum:densitylbd}. Suppose that assumptions~\ref{chap4_assum:M}-\ref{chap4_assum:compatdomains} and \ref{chap4_assum:cover-num} hold. Let $f$ be the unique viscosity solution to \eqref{chap4_eikonal-eq} and $f^{n}$ be a solution to \eqref{chap4_cauchy-graph-fw}. Take $\Delta t = \e_n^{1+\zeta}$ where $\e_n$ is as given in \eqref{chap4_eq:epsn}. Then, there exist  two constants $K > 0$ and $K_2 > 0$ that depend only on $C_\M$, $a$, $\diam(\M)$, $\norm{f_0}_{L^\infty(\M)}$, $\norm{P}_{L^\infty(\M\setminus\Gamma)}$, $\Lip{f_0}$, $\Lip{f_0^n}$, $\Lip{P}$, $\Lip{\widetilde{P}}$, $c_\eta$, $C_\eta$, $L_\eta$, $\norm{\eta}_{L^\infty}$ and $\nu$, and for any $\tau > 0$, there exists $n(\tau) \in \N$ such that for $n \geq n(\tau)$,
\begin{align*}
\normL{f^n-f}_{L^\infty\pa{V_n \times \set{0,\ldots,t_{N_T}}}}  
&\leq K(T+1)\Bigg((1+\tau)^{\frac{\min(\nu,\xi,1/2,\zeta)}{(1+\nu)m^*}}\pa{\frac{\log n}{n}}^{\frac{\min(\nu,\xi,1/2,\zeta)}{(1+\nu)m^*}} \\
&\qquad\qquad\qquad+ \normL{P - \widetilde{P}}_{L^\infty(V_n\setminus\Gamma_n)} \Bigg)
+ \normL{f_0^n-f_0}_{L^\infty(V_n)} .
\end{align*}
with probability at least $1-K_2 n^{-\tau}$. In particular, if $\e_n$ is chosen with $\tau > 1$, $\widetilde{P}=P$ on $V_n\setminus\Gamma_n$ and $f_0^n=f_0$ on $V_n$, then
\begin{align*}
\lim_{n \to +\infty} \normL{f^n-f}_{L^\infty\pa{V_n \times \set{0,\ldots,t_{N_T}}}}  = 0 \quad \text{almost surely}.
\end{align*}
The best convergence rate is $O\pa{\frac{\log n}{n}}^{\frac{1}{3m^*}}$ obtained for $\nu=1/2$, $\xi=1/2$ and $\zeta = 1/2$.
\end{thm}

\begin{proof}
The proof of this theorem is similar to the one of \cite[Theorem 4.3]{fadili2023limits} and we skip it.
% For the above error bound, combine Theorem~\ref{chap4_thm:discrete-fw-estimate} and Lemma~\ref{chap4_lem:compatassum-graph} and observe that 
% \[
% \distH(\Gamma,\Gamma_n) \leq a\e_n^{1+\nu}/2 = o\pa{\e_n^{\min(\nu,1/2)}}
% \]
%  with the same probability. For the last claim, we have for any $\delta > 0$ and $n$ large enough, that 
% \begin{align*}
% &\mathbb{P} \pa{\normL{f^n-f}_{L^\infty\pa{V_n \times \set{0,\ldots,t_{N_T}}}} > \delta} \\ 
% &\leq \mathbb{P}\Bigg(\normL{f^n-f}_{L^\infty\pa{V_n \times \set{0,\ldots,t_{N_T}}}} > K(T+1)\Bigg((1+\tau)^{\frac{\min(\nu,1/2,\zeta)}{(1+\nu)m^*}}\pa{\frac{\log n}{n}}^{\frac{\min(\nu,1/2,\zeta)}{(1+\nu)m^*}} \Bigg) \Bigg) \\
% &\leq K_2 n^{-\tau},
% \end{align*}
% and the right-hand side is summable for $\tau > 1$. The claim then follows using the (first) Borel-Cantelli lemma.
\end{proof}

% \begin{rem}\label{chap4_rem:graph-bw-estimateexpect}
% One can also easily derive from the above a bound in expectation. Let $\mathbf{1}_{\En}=1$ if $\En$ holds and $0$ otherwise. We then have, for $n$ large enough,
% \begin{align*}
% &\E\pa{\normL{f^n-f}_{L^\infty\pa{V_n \times \set{0,\ldots,t_{N_T}}}}} \\
% &= \E\pa{\normL{f^n-f}_{L^\infty\pa{V_n \times \set{0,\ldots,t_{N_T}}}}\Big|\mathbf{1}_{\En}=1} \Pr\pa{\En} 
% + \E\pa{\normL{f^n-f}_{L^\infty\pa{V_n \times \set{0,\ldots,t_{N_T}}}}\Big|\mathbf{1}_{\En}=0} (1-\Pr\pa{\En}) \\
% &\leq 
% K(T+1)\Bigg((1+\tau)^{\frac{\min(\nu,\xi,1/2,\zeta)}{(1+\nu)m^*}}\pa{\frac{\log n}{n}}^{\frac{\min(\nu,\xi,1/2,\zeta)}{(1+\nu)m^*}} \Bigg) 
% + T\normL{P - \widetilde{P}}_{L^\infty(V_n\setminus\Gamma_n)} + \normL{f^n_0-f_0}_{L^\infty(V_n)} 
% + O(n^{-\tau}),
% \end{align*}
% where we have used Theorem~\ref{chap4_thm:graph-bw-estimate} and that $f$ and $f^n$ are bounded. When $\widetilde{P}=P$ and $f_0^n=f_0$ on $V_n\setminus\Gamma_n$ and $V_n$ respectively, we again conclude that $\E\pa{\normL{f^n-f}_{L^\infty\pa{V_n \times \set{0,\ldots,t_{N_T}}}}} \to 0$ as $n \to +\infty$. From this, we also get convergence in probability\footnote{This is also an immediate consequence of almost sure convergence when $\tau > 1$.} thanks to Markov's inequality.
% \end{rem}

\appendix

%\section{Smoothness of the distance function}
%\label{chap4_sec:distsmmoth}

% We will denote by $Cut(x)$ the set of all cut points of $x$ along all geodesics that start from $x$, and call it the cut locus of $x$. Suppose $\g(t_0)$ is the cut point of $x = \g(0)$ along a normal geodesic $\g$, then   $\g(t_0)$ is the first conjugate point of $p$ along $\g$ or $\g(t_0)$ is the first point along $\g$ so that there exists another normal geodesic $\sigma \neq \g$ from $x$ to $\g(t_0)$ with length $L(\sigma) = t_0 = L(\gamma|_{[0,t_0]})$.
% Moreover, the function $d_M(x,.)$ is smooth on $M \backslash Cut(x) \cup \{x\}$.
%\begin{proof}
%To calculate the gradient of the distance $d_\M(x,.)$ at $y$, let $\gamma$ be the unique normal minimizing geodesic from $x$ to $y$, let $X \in T_{y}\M$ be an arbitrary unit vector and let $\sigma(s)$ be a smooth curve in $\M \setminus Cut(x) \cup \{x\}$ such that $y = \sigma(0)$ and $\sigma ' (0) = X$.  Now we consider the variation of $\gamma$ so that $\g_s$ be the unique minimizing geodesic from $x $ to $\sigma(s)$. Observe that the variation field vector of this variation at the point $y$ is exactly $X$. So according to the first variation formula,
%$$
%\langle \nabla d_\M(x,y), X \rangle =\left. \frac{d}{ds} \right|_{s=0} d_\M(x,\sigma(s)) = \left. \frac{d}{ds} \right|_{s=0} L( \gamma_s)= \langle \g'(d_\M(x,y)),X \rangle.
%$$
%\end{proof}

%proposition 5.1
\section{Well-posedness and regularity properties of \eqref{chap4_eikonal-eq-discrete-fw}} 
\label{chap4_sec:cauchy-J-discrete}

We recall the notions of discrete sub- and super-solution defined in \cite{fadili2023limits}. 
\begin{defn}[Discrete sub- and super-solution]\label{chap4_def:discretesolution}
We say that $f^\e$ is a sub-solution to \eqref{chap4_eikonal-eq-discrete-fw} if for all $(x,t) \in   (\widetilde \M \setminus\widetilde  \Gamma) \times \set{t_1,\ldots,t_{N_T}}$
\[
\frac{f^\e(x,t) - f^{\e}(x,t-\Delta t)}{\Delta t} + \abs{\nabla_{\eta_\e}^- f^{\e}(x,t-\Delta t)}_\infty \leq \widetilde{P}(x) ,
\]
and if for all $(x,t) \in \partial \widetilde \M_{N_T}$,
\[
f^{\e}(x,t) \le f_0^\e(x) .
\]
In the same way, we say that $f^\e$ is a super-solution to \eqref{chap4_eikonal-eq-discrete-fw} if for all $(x,t) \in  (\widetilde \M \setminus\widetilde  \Gamma) \times \set{t_1,\ldots,t_{N_T}}$
\[
\frac{f^\e(x,t) - f^{\e}(x,t-\Delta t)}{\Delta t} + \abs{\nabla_{\eta_\e}^- f^{\e}(x,t-\Delta t)}_\infty \geq \widetilde{P}(x) ,
\]
and if for all $(x,t) \in \partial \widetilde \M_{N_T}$,
\[
f^{\e}(x,t) \ge f_0^{\e}(x).
\]
$f^\e$ is a discrete solution to \eqref{chap4_eikonal-eq-discrete-fw} if it is both a discrete sub-solution and super-solution.
\end{defn}

We recall the comparison principle, proved in \cite[Lemma B.2]{fadili2023limits}

\begin{lem}[Comparison principle for the scheme \eqref{chap4_eikonal-eq-discrete-fw}]\label{chap4_lem-comparison}
Assume that \ref{chap4_assum:M}-\ref{chap4_assum:gamma} and \ref{chap4_etapos} hold, and that $f^\e,g^\e$ are respectively bounded sub- and super-solution to \eqref{chap4_eikonal-eq-discrete-fw}. Assume also that the CFL condition \eqref{chap4_cond-delta-t} holds. Then
\begin{eqnarray}
\sup_{\widetilde{\M} \times  \set{0,\ldots,t_{N_T}}}\pa{f^\e - g^\e} \leq \sup_{ \widetilde{\Gamma} \times  \set{t_1,\ldots,t_{N_T}} \cup \widetilde{\M} \times \set{0}} |f^\e - g^\e |.
\end{eqnarray}
\end{lem}

We now establish the existence and the regularity properties of a discrete solution.

\begin{lem}[Existence and Lipschitz regularity properties in time and space for the scheme \eqref{chap4_eikonal-eq-discrete-fw}]\label{chap4_lem-existence-lip-fw}
Assume that assumptions~\ref{chap4_assum:M}-\ref{chap4_assum:f_0}, \ref{chap4_assum:tilde-d}-\ref{chap4_eta:dec} and \ref{chap4_assum:psisssol-J}--\ref{chap4_assum:compatdomains} hold. Then there exists a discrete solution  $f^\e$ to \eqref{chap4_eikonal-eq-discrete-fw} and for all $(x,y) \in \widetilde{\M} \times \widetilde{\M}$ and $t \in \set{t_1,\dots,t_{N_T}}$, the following holds
\begin{align}
\abs{f^\e(x,t) - f^\e(x,t-\Delta t)} &\leq L \Delta t, \label{chap4_lip-t-fw} \\
\abs{f^\e(x,t) - f^\e(y,t)} &\leq K\pa{d_\M(x,y)+\e} , \label{chap4_eq:globlip-space-J-fw}
\end{align}
where $L = \Lip{f_0^\e} + \norm{\widetilde{P}}_{L^\infty(\widetilde{\M}\setminus\widetilde{\Gamma})}$ and $K = 4 a^{-1} \max{\pa{(a+C_\M) \norm{\widetilde{P}}_{L^\infty(\widetilde{\M} \setminus \widetilde{\G})},c_\eta^{-1} C_\eta (L+ \norm{\widetilde{P}}_{L^\infty(\widetilde{\M} \setminus \widetilde{\G})})}}$.
\end{lem}
\begin{proof}
The proof is the same as the one in \cite[Lemma B.3]{fadili2023limits} and we skip it.
\end{proof}

\section{Covering number of $\M$} \label{chap4_sec:prooflemcompatassum-graph}
We will give a lower bound of the covering number of the manifold $\M$ that will be needed in the proof of Lemma \ref{chap4_lem:compatassum-graph}.
We first recall that a $\delta$-covering number ${N}(\M,\delta)$ of a manifold $\M$ is the smallest number of (geodesic) balls of radius $\delta$ needed to cover $\M$. The condition of the radius is linked with the geometry of the manifold, including its curvature. One of the ways to describe the curvature of Riemannian manifolds is the sectional curvature $K(\sigma_p)$, which depends on a two-dimensional linear subspace $\sigma_p$ of the tangent space at a point $p$ on the manifold. More precisely, given two linearly independent tangent vectors at the same point, $u$ and $v$, we can define 
$$K(u,v)=\frac{\langle R(u,v)v,u\rangle}{\langle u,u\rangle \langle v,v\rangle -\langle u,v \rangle^2}.$$
Here $R$ is the Riemannian curvature tensor, defined by 
$R(u,v)w=\nabla_u \nabla_v w -\nabla_v \nabla_u w -\nabla_{[u,v]}w$.
For example, the sectional curvature of a $n-$sphere of radius $r$ is $K=1/r^2$. 
As for the injectivity radius at a point $x\in \M$, it is the supremum of all positive real numbers for which the exponential map is a diffeomorphism when restricted to the open ball of radius $r$ centered at $x$ in the tangent space $T_x\M$. Moreover, the injectivity radius of $\M$ is the infimum of the injectivity radii at all points in $\M$. In \cite{loubes2008kernel}, the authors have shown that if $\delta$ satisfies assumption \ref{chap4_assum:cover-num}, then there exists a strictly positive constant $C$ such that 
\begin{equation*}
    {N} (\M,\d) \leq C \vol(\M) \delta^{-m^*}.
\end{equation*}

\nomenclature{$\Exp$}{exponential map}

% \printindex
\vspace{1cm}
\paragraph{\textbf{Acknowledgment.}}
This research was partially funded by l’Agence Nationale de la Recherche (ANR), project ANR-22-CE40-0010. 
For the purpose of open access, the authors have applied a CC-BY public copyright license to any Author Accepted Manuscript (AAM) version arising from this submission.

\end{document}